\documentclass[11pt]{amsart}
\usepackage{amsmath}
\usepackage{latexsym}
\usepackage{amsfonts}
\usepackage{amssymb}

\newcommand{\beqa}{\begin{eqnarray*}}
\newcommand{\eeqa}{\end{eqnarray*}}
\newcommand{\beqn}{\begin{eqnarray}}
\newcommand{\eeqn}{\end{eqnarray}}

\newcommand{\iy}{\infty}
\newcommand{\lt}{\left}
\newcommand{\rt}{\right}

\newcommand{\bQ}{\mathbb Q}

\newcommand{\C}{\mathbb C}

\newcommand{\R}{\mathbb R}

\newcommand{\N}{\mathbb N}
\newcommand{\M}{\mathbb M}

\newcommand{\mcA}{\mathcal A}
\newcommand{\mcH}{\mathcal H}
\newcommand{\mcB}{\mathcal B}

\newcommand{\mcF}{\mathcal F}

\newcommand{\mfB}{\mathfrak B}
\newcommand{\mfL}{\mathfrak L}
\newcommand{\mfO}{\mathfrak O}
\newcommand{\mfX}{\mathfrak X}
\newcommand{\mfR}{\mathfrak R}
\newcommand{\mfT}{\mathfrak T}

\newcommand{\f}{\frac}
\newcommand{\tf}{\tfrac}

\newcommand{\al}{\alpha}

\newcommand{\G}{\Gamma}
\newcommand{\g}{\gamma}
\newcommand{\e}{\varepsilon}

\newcommand{\De}{\Delta}
\newcommand{\la}{\lambda}

\newcommand{\Om}{\Omega}

\newcommand{\s}{\sigma}
\newcounter{cnt1}
\newcounter{cnt2}
\newcounter{cnt3}
\newcommand{\blr}{\begin{list}{$($\roman{cnt1}$)$}
 {\usecounter{cnt1} \setlength{\topsep}{0pt}
 \setlength{\itemsep}{0pt}}}
\newcommand{\bla}{\begin{list}{$($\alph{cnt2}$)$}
 {\usecounter{cnt2} \setlength{\topsep}{0pt}
 \setlength{\itemsep}{0pt}}}
\newcommand{\bln}{\begin{list}{$($\arabic{cnt3}$)$}
 {\usecounter{cnt3} \setlength{\topsep}{0pt}
 \setlength{\itemsep}{0pt}}}
\newcommand{\el}{\end{list}}
\newtheorem{thm}{Theorem}[section]
\newtheorem{lem}[thm]{Lemma}
\newtheorem{cor}[thm]{Corollary}
\newtheorem{ex}[thm]{Example}

\newtheorem{Def}[thm]{Definition}

\newtheorem{rem}[thm]{Remark}
\newcommand{\Rem}{\begin{rem} \rm}
\newcommand{\bdfn}{\begin{Def} \rm}
\newcommand{\edfn}{\end{Def}}

\newcommand{\ba}{\begin{array}}
\newcommand{\ea}{\end{array}}

\numberwithin{equation}{section}
\date{}
\begin{document}
\title{\bf Constructive Analysis in Infinitely many variables}
\author[Gill]{Tepper L. Gill}
\address[Tepper L. Gill]{ Department of Mathematics, Physics and E\&CE, Howard University\\
Washington DC 20059 \\ USA, {\it E-mail~:} {\tt tgill@howard.edu}}
\author[Pantsulaia]{G. R. Pantsulaia}
\address[Gogi R. Pantsulaia] { Department of Mathematics  \\ Georgian Technical University \\ Tbilisi  0175 \\ Georgia; { I. Vekua Institute of Applied Mathematics  \\ Tbilisi State University \\ Tbilisi  0143 \\ Georgia} {\it E-mail~:} {\tt
g:pantsulaia@gtu:ge}}
\author[Zachary]{W. W. Zachary*}
\address[Woodford W. Zachary]{ Department of Mathematics and E\&CE \\ Howard
University\\ Washington DC 20059 \\ USA, {\it E-mail~:} {\tt
wwzachary@earthlink.net}}
\thanks{*deceased}
\date{}
\subjclass{Primary (45) Secondary(46) }
\maketitle
\begin{abstract}  In this paper we investigate the foundations for analysis in infinitely-many (independent) variables.   We give  a  topological  approach to the construction of the regular $\s$-finite Kirtadze-Pantsulaia measure on $\R^\iy$ (the usual completion of  the Yamasaki-Kharazishvili measure),  which is an infinite dimensional version of the classical method  of  constructing  Lebesgue measure on $\R^n$ (see \cite{YA1}, \cite{KH} and \cite{KP2}). First we show that von Neumann's theory of infinite tensor product Hilbert spaces already implies that a natural version of Lebesgue measure must exist on $\R^{\iy}$.  Using this insight, we define the canonical version of $L^2[\R^{\iy}, \la_{\iy}]$, which allows us to construct Lebesgue measure on $\R^{\iy}$ and analogues of Lebesgue and Gaussian measure for every separable Banach space with a Schauder basis.  When $\mcH$ is a Hilbert space and $\la_{\mcH}$ is Lebesgue measure restricted to $\mcH$, we define sums and products of unbounded operators and the Gaussian density for $L^2[\mcH, \la_{\mcH}]$.   We show that the Fourier transform induces two different versions of the Pontryagin duality theory.  An interesting new result is that the character group changes on infinite dimensional spaces when the Fourier transform is treated as an operator.  Since our construction provides a complete $\s$-finite measure space, the abstract version of Fubini's theorem allows us to extend Young's inequality to every separable Banach space with a Schauder basis.   We also give constructive examples of partial differential operators in infinitely many variables and briefly discuss the famous partial differential equation derived by Phillip Duncan Thompson \cite{PDT}, on  infinite-dimensional phase space to represent an ensemble of randomly forced two-dimensional viscous flows. 
\end{abstract}
\tableofcontents
\section*{\bf Introduction}
On finite-dimensional space it is useful to think of Lebesgue  measure in terms of geometric objects (e.g.,volume, surface area, etc.).  Thus, it is natural to expect that this measure will leave these objects invariant under translations and rotations, so that rotational and translational invariance is an intrinsic property of Lebesgue  measure.  However, we then find ourselves disappointed when we try to use this property to help define Lebesgue measure on $\mathbb{R}^{\infty}$.  A more fundamental problem is that the natural Borel algebra for $\R^{\iy}, \; {\mcB}[\R^{\iy}]$, does not allow an outer measure (since the measure of any open set is infinite).  

The lack of any definitive understanding of the cause for this lack of invariance on $\mathbb{R}^{\infty}$ has led some researchers  to believe that it is not possible to have a reasonable version of Lebesgue measure on $\mathbb{R}^{\infty}$ (see, for example, DaPrato \cite{DP} or Bakhtin and Mattingly \cite{BM}).   In many applications, the study of infinite dimensional analysis is restricted to separable Hilbert spaces, using Gaussian measure as a replacement for (the supposed nonexistent) Lebesgue measure.   In some cases the Hilbert space structure arises as a natural state space for the modeling of systems.   In other cases, both the Hilbert spaces and probability measures are imposed for mathematical convenience and are physically artificial and limiting. However, all reasonable models of infinite dimensional (physical) systems require some functional constraint on the effects of all but a finite number of variables. Thus, what is needed, in general, is the imposition of constraints on the functions while preserving the modeling freedom associated with  infinitely-many independent variables (in some well-defined sense).  Any attempt to solve this problem necessarily implies a theory of Lebsegue measure on $\mathbb{R}^{\infty}$.  

Even if a reasonable theory of Lebsegue measure on $\mathbb{R}^{\infty}$ exists, this is not sufficient to make it useful in engineering and science.  In addition, all the tools developed for finite-dimensional analysis, differential operators, Fourier transforms,  etc are also required.  Furthermore, researchers  need operational control over the convergence properties of these tools.  In particular, one must be able to approximate an infinite-dimensional problem as a natural limit of the finite-dimensional case in a manner that lends itself to computational implementation.  This implies that a useful approach also has a well-developed theory of convergence for infinite sums and products of unbounded linear operators.
\section*{Historical Background}
Research into the general problem of Lebesgue measure on infinite-dimensional vector spaces and $\R^{\iy}$ in particular, has a long and varied past, with participants living in a number of different countries, during times when scientific communication was constrained by war, isolation and/or national competition.  These conditions allowed quite a bit of misinformation and folklore to grow up around the subject, so that even experts may have a limited view of the history.   Our own experience suggest that at least a brief survey of some important events is in order.  (We do not claim completeness and apologize in advance if we fail to mention equally important contributions.)

Early studies in infinite dimensional analysis focused on the foundations of probability theory and had a broad base of participation.  However, the major inputs were made by researchers in Poland, Russia, and France, with later contributions from the US.  
The first important advance of the general theory was made in 1933 when  Haar \cite{HA} proved the following theorem:
\begin{thm}On every locally compact abelian group $G$ there exists a non-negative regular measure $m$ (Haar measure) on $G$, which is not identically zero and is translation invariant.  That is, $m(A+x)=m(A)$ for every $x \in G$ and every Borel set $A$ in $G$. 
\end{thm}
This theorem stimulated interest in the subject and von Neumann \cite{VN1}  proved that it is the only locally finite left-invariant Borel measure on the group (uniqueness up to a mulitplicative constant).  Weil \cite{WE} developed an axiomatic approach to the subject,  made a number of important refinements and, proved the ``Inverse Weil theorem" (in moderm terms):
\begin{thm}If $G$ is a (separable) topological  group and $m$ is a  translation invariant Borel measure on $G$, then it is always possible to define an equivalent locally compact topology on $G$.
\end{thm}
In 1946,  Oxtoby \cite{OX} initiated the study  of translation-invariant Borel measures on Polish groups (i.e., complete separable metric groups).  In this paper, Oxtoby provides a proof of the following result which he attributes to Ulam:
\begin{thm}
Let $G$ be any complete separable metric group which is not locally compact, and let $m$ be any left-invariant Borel measure in G. Then every neighborhood contains an uncountable number of disjoint mutually congruent sets of equal finite positive measure. 
\end{thm}
Stated another way, he proved that 
\begin{thm} There always exists a left-invariant Borel measure on any Polish group which assigns positive finite measure to at least one set and vanishes on singletons.  However, a locally finite measure is possible if and only if the group is locally compact.
\end{thm} 
(In 1967, Vershik \cite{V} proved a related result for probability measures.)
Apparently uninformed of Oxtoby's work, In 1959 Sudakov \cite{SU} independently proved a special case of Theorem 0.4: {\it If $\R^\iy$ is regarded as a linear topological space, then there does not exist a $\s$-finite translation-invariant Borel measure for $\R^\iy$}.  In 1964, Elliott and  Morse \cite{EM} developed a general theory of translation invariant product measures (non-$\s$-finite) and, in 1965, C. C. Moore \cite{MO} initiated the study of measures that are translation invariant with respect to vectors in $\R_0^\iy$ (i.e., the set of sequences that are zero except for a finite number of terms).  This work was extended and refined by Hill \cite{HI} in 1971.

Motivated by Kakutani's work on infinite product measures \cite{KA}, a number of young Japanese researchers entered the field.  In 1973, Hamachi \cite{HA} made major improvements on  Hill's work which, indirectly suggested the problem of identifying the largest group $\mfT$, of admissible translations in the sense of invariance for any $\s$-finite Borel measure  $\mu$  on $\R^\iy$ which assigns the value of one to $[-\tf{1}{2}, \tf{1}{2}]^{{\aleph _o}}$  and is metrically transitivity with respect to  $\R_0^\iy$ (equivalently, for  each  $A$  with  $\mu(A)>0$, there is a sequence  $(h_k) \in \R_0^\iy$  such that  $\mu(\R^\iy  \setminus \cup_{k =1}^{\iy}(A+h_k))=0$).   

Yamasaki \cite{YA1} solved this problem in 1980.  Unaware of the Yamasaki's proof, \newline Kharazishvili independently solved the same problem in 1984.  In 1991 Kirtadze and Pantsulaia \cite{KP1}  provided yet another solution (see also Pantsulaia \cite{PA}). Finally, In 2007, Kirtadze and Pantsulaia  proved that, if $\overline{\mu}$ is the completion  of the  measure  $\mu$, then:  (see \cite{KP2})  
\begin{thm}The measure $\overline{\mu}$  is  the unique regular $\s$-finite measure on $\R^\iy$    (uniqueness up to a mulitplicative constant),  which is assigns the  value  one to the  set  $[-\tf{1}{2}, \tf{1}{2}]^{{\aleph _o}}$, is invariant under translations from the group $\ell_1 $ and  has  the metrically transitivity property with respect to  $\ell_1 $. 
\end{thm}

In the mean time, in 1991 Baker  \cite{BA1}, unaware of the Elliott-Morse measures,  dropped the requirement that the measure be $\s$-finite and constructed a translation invariant measure, $\nu$, on $\R^{\iy}$ (see also Baker (2004),  \cite{BA2}). In 1992, Ritter and Hewitt \cite{RH} constructed a translation invariant measure related to that of Elliott Morse.

Starting in 2007, A. M. Vershik (see \cite{V1}, \cite{V2}, \cite{V3} and references contained therein) started an investigation of an infinite-dimensional analogue of Lebesgue measure that is constructed in a different manner than that studied in the previous papers.  Roughly stated, he considers the weak limit as $n \rightarrow \infty$ of invariant measures on certain homogeneous spaces (hypersurfaces of high dimension) of the Cartan subgroup of the Lie groups ${\bf{SL}}(n, \R)$ (i.e., the subgroups of diagonal matrices with unit determinant). Vershik's measure is also unique and invariant under the multiplicative group of positive functions, suggesting that a logarithmic transformation may lead to a version of the measure in this paper.  (The paper of Vandev \cite{VA} should also be consulted.) 
\subsection*{Purpose}
The purpose of this paper is to show that a minor change in the way we represent $\mathbb{R}^{\infty}$ makes it possible to construct a $\s$-finite regular version of Lebesgue measure using basic methods of measure theory from $\R^n$.  Since the measure is regular, it turns out to be the Kirtadze and Pantsulaia \cite{KP1} measure, which is unique (see Theorem 0.5). Using our approach, we construct an analogue of both Lebesgue and Gaussian measure  (countably additive) on every (classical) separable Banach space with a Schauder basis.  The version of Gaussian measure constructed is also rotationally invariant (a property not shared by Wiener measure). This approach also allows us to satisfy all the requirements of a useful infinite dimensional  theory.   
\subsection*{Summary}
In the first section, we show how von Neumann's infinite tensor product Hilbert space theory implies that a natural version of Lebesgue measure must exist on $\R^{\iy}$ and points to a possible approach.  In the first part of  Section 2, we show that a slight change in thinking about the cause for problems with unbounded measures on $\mathbb{R}^{\infty}$  makes the construction of Lebesgue measure not only possible, but no more difficult then the same construction on $\mathbb{R}^{n}$. (We denote it by $\mathbb{R}_I^{\iy}$, for reasons that are discussed in this section.)  We also provide  natural analogues of Lebesgue and Gaussian measure for every separable Banach space with a Schauder basis and show that $\ell_1$ is the maximal translation invariant subspace.  In the last part of Section 2, we show that $\ell_2$ is the maximal rotation invariant subspace.   In Section 3, we study the convergence properties of infinite sums and products of bounded and unbounded  linear operators.  In Section 4, we investigate some of the function spaces over $\R_I^\iy$ and in Section 5, we discuss Fourier transforms and Pontryagin duality theory for Banach spaces.  A major result is that there are two different extensions of the Pontrjagin Duality theory for infinite dimensional spaces. In this section, we also show that our theory allows us to extend Young's inequality to ever separable Banach space with a Schauder basis.  In Section 6, we give some constructive examples of partial differential operators in infinitely many variables.  This allows us to briefly  discuss the famous partial differential equation derived by Phillip Duncan Thompson \cite{PDT}, on  infinite-dimensional phase space to represent an ensemble of randomly forced two-dimensional viscous flows.
\section{Why $\la_{\iy}$ Must Exist}
In order to see that some reasonable version of Lebesgue measure must exist, we need to review von Neumann's  infinite tensor product Hilbert space theory \cite{VN2}.   To do this, we first define infinite products of complex numbers. (There are a number of other possibilities, see \cite{GU} and \cite{PA}, pg. 272-274.) In order to avoid trivialities, we always assume that, in any product, all terms are nonzero. 
\begin{Def} If $\{z_i \}$ is a sequence of complex numbers indexed by $i \in \N$ (the natural numbers),  
\begin{enumerate}
\item  We say that the product $\prod\nolimits_{i \in \N} {z_i }$ is convergent with limit $z$ if, for every $\varepsilon  > 0$, there is a finite set $J(\varepsilon)$ such that, for all finite sets $J \subset \N$, with   $J(\varepsilon)\subset J$, we have $\left| {\prod\nolimits_{i \in J} {z_i }  - z} \right| < \varepsilon$.
\item We say that the product $\prod\nolimits_{i \in \N} {z_i  }$ is quasi-convergent if $
\prod\nolimits_{i \in \N} {\left| {z_i  } \right|}$ is convergent.  (If the product is quasi-convergent, but not convergent, we assign it the value zero.)
\end{enumerate}
\end{Def}
We note that 
\beqn
0 < \left| {\prod\nolimits_{i \in \N} {z_i } } \right| < \infty \; {\text {if and only if}} \; \sum\nolimits_{i  \in \N} {\left| {1 - z_i } \right|}  < \infty.  
\eeqn
Let ${\mcH}_i=L^2[\R, {\la}]$ for each $i \in \N$ and let $\mcH_ \otimes ^2  = \hat  \otimes _{i = 1}^\infty  L^2 [\mathbb{R},\lambda ]$ be the infinite tensor product of von Neumann.  To see what this object looks like:
 \begin{Def}
 Let $g  = \mathop  \otimes \limits_{i \in \N} g _{i}$ and $ 
h  = \mathop  \otimes \limits_{i \in \N} h _{i} $ be in $
 {\mcH}_ \otimes ^2 $. 
\begin{enumerate}
\item We say that $g$ is strongly equivalent to $h$ ($g  \equiv^s h $) if and only if $
\sum\limits_{i \in \N} {\left| {1 - \left\langle {g _{i} ,h _{i} } \right\rangle _{i} } \right|}  < \infty \;.$
\item We say that $g$ is weakly equivalent to $h $  ($g  \equiv ^w h $) if and only if $
\sum\limits_{i \in \N} {\left| {1 - \left| {\left\langle {g _i  ,h _i  } \right\rangle _i  } \right|\,} \right|}  < \infty. $
\end{enumerate}
\end{Def}
Proofs of the following may be found in von Neumann \cite{VN2} (see also \cite{GZ}, \cite{GZ1}).
\begin{lem} We have $g  \equiv ^w h$ if and only if there exist $z_i  ,\;|\,z_i  \,| = 1$, such that $
\mathop  \otimes \limits_{i  \in \N} z_i  g _i   \equiv ^s \mathop  \otimes \limits_{i  \in \N} h _i  $.
\end{lem} 
\begin{thm}
The relations defined above are equivalence relations on ${\mcH}_ \otimes ^2 $, which decomposes ${\mcH}_ \otimes ^2$ into disjoint  equivalence classes (orthogonal subspaces).
\end{thm}  
\begin{Def} 
For $
g  = \mathop  \otimes \limits_{i \in \N} g _{i}  \in {\mcH}_ \otimes ^2 $, we define ${\mcH}_ \otimes ^2 (g )$ to be the closed subspace generated by the span of all $
h  \equiv^s g $ and we call it the strong partial tensor product space generated by the vector $g$. (von Neumann called it an incomplete tensor product space.)
\end{Def}
\begin{thm}
For the partial tensor product spaces, we have the following:  
\begin{enumerate}
\item 
If $h _{i}  \ne g _{i} $ occurs for at most a finite number of ${i}$, then $
h  = \mathop  \otimes \limits_{i \in \N} h _{i}  \equiv^s g  = \mathop  \otimes \limits_{i \in \N} g _{i} $.
\item 
The space $ {\mcH}_ \otimes ^2 (g )$ is the closure of the linear span of $h  = \mathop  \otimes \limits_{i \in \N} h _{i} $ such that $h _{i}  \ne g _{i} $ occurs for at most a finite number of ${i}$.
\item 
If $g  =  \otimes _{i \in \N} g _{i}  $
 and $h  =  \otimes _{i \in \N} h _{i}  $
 are in different equivalence classes of $
{\mcH}_ \otimes ^2$, then $
\left( {g ,h } \right)_ \otimes   = \prod _{i \in \N} \left\langle {g _{i} ,h _{i} } \right\rangle _{i}  = 0$.
\item
$
{\mcH}_ \otimes ^2 (g )^w  = \mathop  \oplus \limits_{h  \equiv ^w g } {\kern 1pt} \left[ {\mcH}_ \otimes ^2 (h )^s  \right].
$
\item For each $g,\; {\mcH}_ \otimes ^2 (g)^s$ is a separable Hilbert space.

\item For each $g,\; {\mcH}_ \otimes ^2 (g )^w$ is not a separable Hilbert space. 
\end{enumerate}
\end{thm} 
It follows from (6) that $\mcH_ \otimes ^2  = \hat  \otimes _{i = 1}^\infty  L^2 [\mathbb{R},\lambda ]$ is not a separable Hilbert space. 

From (5), we see that it is reasonable to define $L^2 [\R^{\infty}, {\la}_{\infty}]=\mcH_ \otimes ^2(h)^s$, for some $h=\otimes_{i=1}^{\iy }h_i$.  This definition is ambiguous, but, in most applications, the particular version does not matter.  To remove the  ambiguity, we should identify a canonical version of $h=\otimes_{i=1}^{\iy }h_i$.  Any reasonable version of  $\la_{\iy}$ should satisfy  $\la_{\iy}(I_0)=1$, where $I=[\tf{-1}{2}, \tf{1}{2}]$ and  $I_0=\times_{i=1}^{\iy}I$.
\begin{Def} If $\chi_I$ is the indicator function for $I$ and $h_i =\chi_I$, we set $h=\otimes_{i=1}^{\iy}h_i$.  We define the canonical version of $L^2 [\R^{\infty}, {\la}_{\infty}]= L^2 [\R^{\infty}, {\la}_{\infty}](h)^s$.
\end{Def}
\section{{ Lebesgue Measure on $\mathbb{R}_I^{\infty}$}} 
\subsection{The Construction} 
We now have the problem of identifying the measure space associated with  $L^2 [\R^{\infty}, {\la}_{\infty}](h)^s$.  In the historical approach to the construction of infinite products of measures $\{ \mu_k,\;  k \in \mathbb{N} \}$ on ${\R}^{\iy}$, the chosen topology defines open sets to be the (cartesian) product of an arbitrary finite number of open sets in $\mathbb{R}$, while the remaining infinite number are copies of $\mathbb{R}$ (cylindrical sets).  The success of Kolmogorov's work on the foundations of probability theory naturally led to the condition that $\mu_k(\mathbb{R})$ be finite for all but a finite number of $k$ (see \cite{KO}).   Thus, any attempt to construct Lebesgue measure via this approach starts out a failure in the beginning.   However,  Kolmogorov's approach is not the only way to induce a total measure of one for the spaces under consideration.  

Our definition of the canonical version of $L^2 [\R^{\infty}, {\la}_{\infty}]$ offers another approach. To see how, consider a simple extension of the theory on $\R$.  Let $I = [ - \tfrac{1}{2},\tfrac{1}{2}]$ and define $\mathbb{R}_I  = \mathbb{R}  \times I_1 $, where $I_1  = \mathop  \times \limits_{i = 2}^\infty I$.  If $\mathfrak{B}(\mathbb{R} )$ is the Borel $\s$-algebra  for $\mathbb{R}$, let $\mathfrak{B}(\mathbb{R}_I )$ be the Borel $\s$-algebra  for $\mathbb{R}_I$.  For each set $A \in \mathfrak{B}(\mathbb{R} )$ with $\la(A) < \infty$, let $A_I$ be the corresponding set in $\mathfrak{B}(\mathbb{R}_I )$, $A_I=A\times I_1$.  We define $\la_{\infty}(A_I)$ by:
\[
\la_{\infty}(A_I)=\lambda (A) \times \prod _{i = 2}^\infty  \lambda (I) = \lambda (A).
\]
We can construct a theory of Lebesgue measure on $\mathbb{R}_I$ that completely parallels that on $\R$.  This suggests that we use Lebesgue measure and replace the (tail end of the) infinite product of copies of $\mathbb{R}$ by infinite products of copies of  $I$.    The purpose of this section is to provide such a construction.  Since we will be studying unbounded measures, for consistency,  we use the following  conventions: $0 \cdot \iy =0$ and $0 \cdot \infty^{\iy}=\iy$.

Recall that $\R^{\iy}$ is the set of all ${\bf x} = \left( {x_1 ,x_2 , \cdots } \right)$, where $x_i \in \R$.  This is a linear space which is not a Banach space.  However, it is a complete metric space with metric given by:
\[
d({\bf x},{\bf y}) = \sum\nolimits_{n = 1}^\infty  {\frac{1}
{{2^{n} }}} \frac{{\left| {x_n  - y_n } \right|}}
{{1 + \left| {x_n  - y_n } \right|}}.
\]
\begin{rem}
$\R^\iy$ is a special case of a Polish space, which  Banach called a Fr{\'e}chet space i.e., a Polish space with a translation invariant metric (see Banach \cite{BA}). The topology generated by $d(\cdot,  \cdot)$ is generally known as the Tychonoff topology.
\end{rem}
For each $n$, define $\mathbb{R}_I^n  = \mathbb{R}^n  \times I_n $, where $I_n  = \mathop  \times \limits_{i = n + 1}^\infty  I$.  
\begin{Def}  If $A_n=A \times I_n,\; B_n=B \times I_n$ are any sets in $\mathbb{R}_I^n$, then we define: 
\begin{enumerate}
\item $A_n \cup B_n= A \cup B \times I_n$,
\item $A_n \cap B_n= A \cap B \times I_n$, and
\item $B_n^c= B^c \times I_n$.
\end{enumerate}
\end{Def}
In order to avoid confusion, we always assume that   $I_0 = \times _{i= 1}^{\infty} I \subset \R_I^1$.      
We can now define the topology for $\mathbb{R}_I^n $ via the following class of open sets:  
\[
{\mathfrak{O}}_n = \left\{ {U_n :U_n  = U \times I_n ,\;U\;{\text{open in }}\mathbb{R}^n } \right\}.
\]
\subsubsection{\bf{Definition of} ${\mathbb{R}}_I^{\infty}$}
It is easy to see that $\mathbb{R}_I^n \subset \mathbb{R}_I^{n+1}$. Since this is an increasing sequence, we can define ${\mathbb{R}'}_I^{\infty}$ by: 
\[
{\mathbb{R}'}_I^{\infty} =\rm{lim}_{n\rightarrow \infty}\mathbb{R}_I^n=  \mathop  \cup \limits_{k =  1}^\infty {\mathbb{R}_I^k}.
\]
Let $\tau_1$ be the topology on ${\mathbb{R}'}_I^{\infty}={\mfX}_1$ induced by the class of open sets $\mfO$ defined by:
\[
\mfO=\bigcup_{n=1}^\iy{{\mathfrak{O}}_n} =\bigcup_{n=1}^\iy{ \left\{ {U_n :U_n  = U \times I_n ,\;U\;{\text{open in }}\mathbb{R}^n } \right\}},
\]
 and let $\tau_2$ be topology on $\R^{\infty} \setminus  {\mathbb{R}'}_I^{\infty}={\mfX}_2$ induced by the metric $d_2$, for which $d_2(x,y)=1, \, x \neq y$  and $d_2(x,y)=0, \, x=y$, for all  $x, y \in {\mfX}_2$.
\begin{Def}We define $({\mathbb{R}}_I^{\infty}, \tau)$ to be the  sum $({\mfX}_1, \tau_1)$ and  $({\mfX}_2, \tau_2)$, so that every open set in $({\mathbb{R}}_I^{\infty}, \tau)$ is union of two disjoint sets $G_1 \cup G_2$, where $G_1$ is open in $({\mfX}_1, \tau_1)$ and $G_2$ is open in $({\mfX}_2, \tau_2)$.
\end{Def} 
It now follows from the above construction that $\R_I^{\infty}=\R^{\infty}$ as sets.  (However, they are not equal as topological spaces.)  The following result shows that convergence in the $\tau$-topology always implies convergence in the Tychonoff topology.
\begin{thm} If $y_k$ converges to $x$ in the $\tau$-topology, then $y_k$ converges to $x$ in the Tychonoff topology.
\end{thm}
\begin{proof}  Case 1.  If $x \in \R^{\infty} \setminus{\mathbb{R}'}_I^{\infty}$ then there is  $N$ such that  $y_k=x$  for all $k >N$. Indeed, for a neighborhood of diameter $\tf{1}{2}$ about $x$, there is a $N$ such that  $d_2(x,y_k)<1/2$ for all  $k > N$.  This means that  $y_k=x$ for $k >N$ ($\{ z:d_2(x,z)<1/2\} $ only contains $x$), so that  $y_k$ converges to $x$ in the Tychonoff topology.

Case 2. If  $x \in {\mathbb{R}'}_I^{\infty}$ and $y_k$ converges to $x$,  then for any neighborhood $U_n \subset \mfO_n$,   there is  $N$ such that or all $k >N, \; y_k \in U_n$. This means that  $y_k \in {\mathbb{R}'}_I^{\infty}$ for  $k >N$, so that  $y_k$ converges to $x$ in the Tychonoff topology.
\end{proof}
\subsubsection{\bf{Definition of} ${\mathfrak{B}}(\mathbb{R}_I^{\infty} )$}
In a similar manner, if  $\mathfrak{B}(\mathbb{R}_I^n )$ is the Borel $\s$-algebra  for $\mathbb{R}_I^n$ (i.e., the smallest $\s$-algebra generated by the ${\mathfrak{O}}_n$), then  $\mathfrak{B}(\mathbb{R}_I^n ) \subset \mathfrak{B}(\mathbb{R}_I^{n+1} )$, so we can define  ${\mathfrak{B}'}(\mathbb{R}_I^{\infty} )$ by:
\[
{\mathfrak{B}'}(\mathbb{R}_I^{\infty} )=\rm{lim}_{n\rightarrow \infty} \mathfrak{B}(\R_I^{n} )=  \mathop  \cup \limits_{k =  1}^\infty \mathfrak{B}(\R_I^{k} ).
\]
If $\mathcal{P}(\cdot)$ denotes a powerset of a set (i.e., $\mathcal{P}(A)=\{ X: X \subseteq A\}$), let  $ \mathfrak{B}(\R_I^{\infty})$ be the smallest $\sigma$-algebra containing ${ \mathfrak{B}'}(\R_I^{\infty}) \cup    \mathcal{P} (R_I^{\infty} \setminus \cup_{n=1}^{\infty}\R_I^{n})$.   (It is obvious that the class $\mfB(\R_{I}^{\infty})$ coincides with Borel $\sigma$-algebra generated by the $\tau$-topology on $\R^{\infty}$.)  From our definition of $\mfB(\R_I^{\infty})$  we see that $\mfB(\R^{\infty}) \subset \mfB(\R_I^{\infty})$ and the containment is proper.
\begin{thm}
$\lambda _\infty  (\cdot)$ is a measure on ${\mathfrak{B}}(\R_I^{n})$, equivalent to $n$-dimensional Lebesgue measure on $\mathbb{R}^n$.
\end{thm}
\begin{proof} If $A= \mathop  \times \limits_{i =  1}^\infty  A_i \in \mathfrak{B}(\mathbb{R}_I^{n})$, then $\la(A_i)=1$ for $i>n$ so that the series $\la_\infty  (A) = \prod\nolimits_{i = 1}^\infty  {\lambda (A_i )}$ always converges.  Furthermore, 
\beqn
0<\la_\infty  (A) = \prod\nolimits_{i = 1}^\infty  {\lambda (A_i )} = \prod\nolimits_{i = 1}^n {\lambda (A_i )}=\la_n(\mathop  \times \limits_{i =  1}^n  A_i).
\eeqn  
Since sets of the type $A =\mathop  \times \limits_{i =  1}^n  A_i$ generate $\mathfrak{B}(\mathbb{R}^n )$, we see that  $\la_\infty(\cdot)$, restricted to $\mathbb{R}_I^{n}$, is equivalent to $\la_n(\cdot)$.
\end{proof}
\begin{cor} The measure $\lambda _\infty  (\cdot)$ is both translationally and rotationally invariant on $(\mathbb{R}_I^{n},\mathfrak{B}[\mathbb{R}_I^{n} ])$.
\end{cor}
\subsection{The Extension  to $\mathbb{R}_I^{\infty}$}
It is not obvious that $\lambda _\infty  (\cdot)$ can be extended to a countably additive measure on ${\mathfrak{B}}(\R_I^{\infty})$.  
\begin{Def} Let  
\[
\De_0=\{K_n= K \times I_n\in \mathfrak{B}(\mathbb{R}_I^{n}) \subset {\mathfrak{B}}(\R_I^{\infty}): n \in \N, \; K\; {\rm is\; compact\; and} \;0<  {\lambda_{\infty} (K_n )}< \infty\},
\]
\[
\De=\{P_N = \bigcup\nolimits_{i = 1}^N K_{n_i}  ,\,N \in \mathbb{N};\,{K_{n_i}} \in \De_0\; {\rm and\;}\lambda _\infty  ({K_{n_l}}  \cap {K_{n_m}} ) = 0,\;l \ne m\}.
\]
\end{Def} 
\begin{Def}  If $P_N \in \De$, we define 
\[
\lambda _\infty  (P_N )= \sum\nolimits_{i = 1}^N \lambda_\infty({K_{n_i}}).
\]
\end{Def}
Since $P_N \in \mathfrak{B}(\mathbb{R}_I^{n})$ for some $n$, and $\lambda _\infty  (\cdot)$ is a measure on $\mathfrak{B}(\mathbb{R}_I^{n})$, the next result follows:
\begin{lem} If $P_{N_1},\; P_{N_2} \in \De$ then:
\begin{enumerate}
\item If $P_{N_1} \subset P_{N_2}$, then $\lambda _\infty  (P_{N_1} ) \le \lambda _\infty  (P_{N_2} )$.
\item If $\lambda _\infty  (P_{N_1} \cap P_{N_2})=0$, then $\lambda _\infty  (P_{N_1} \cup P_{N_2} )=  \lambda _\infty  (P_{N_2} )+  \lambda _\infty  (P_{N_2} )$.
\end{enumerate}
\end{lem}
\begin{Def}  If $G \subset \mathbb{R}_I^{\infty}$ is any open set, we define:
\[
\lambda _\infty  (G) = \mathop {\lim }\limits_{N \to \infty } \sup \left\{ {\lambda _\infty  (P_N ):\,P_N  \in \Delta ,\;P_N  \subset G,\,} \right\}.
\]
\end{Def}
\begin{thm} If $\mathfrak{O}$ is the class of open sets in $ {\mathfrak{B}}(\R_I^{\infty})$, we have:  
\begin{enumerate}
\item $\lambda _\infty(\mathbb{R}_I^{\infty})={\infty}$.
\item If $G_1,\; G_2 \in \mathfrak{O},\; G_1\subset G_2$, then 
$\lambda _\infty(G_1) \le \lambda _\infty(G_2)$.
\item If $\{G_k\} \subset \mathfrak{O}$, then 
\[
\lambda _\infty(\bigcup\nolimits_{k = 1}^{\infty} {G_k }) \le \sum\nolimits_{k = 1}^{\infty}\lambda _\infty({G_k }).
\]
\item If the $G_k$ are disjoint, then 
\[
\lambda _\infty(\bigcup\nolimits_{k = 1}^{\infty} {G_k }) = \sum\nolimits_{k = 1}^{\infty}\lambda _\infty({G_k }).
\] 
\end{enumerate}
\end{thm}
\begin{proof}
The proof of (1) is standard.  To prove (2), observe that 
\[
\left\{ {P_N :\,\;P_N  \subset G_1 } \right\} \subset \left\{ { P'_N :\, P'_N  \subset G_2 } \right\},
\]
so that $\lambda _\infty({G_1 }) \le \lambda _\infty({G_2 })$.  To prove (3),
let $P_N \subset \bigcup\nolimits_{k = 1}^{\infty} {G_k }$.  Since $P_N$ is compact, there is a finite number of the  ${G_k }$  which cover $P_N$, so that $P_N \subset \bigcup\nolimits_{k = 1}^{L} {G_k }$.  Now, for each $G_k$, there is a $P_{N_k} \subset G_k$.  Furthermore, as $P_N$ is arbitrary, we can assume that $P_N ={{P}'}_{N}=\mathop \bigcup\nolimits_{k = 1}^{L}P_{N_k}$.  Since there is an $n$ such that  all $P_{N_k} \in {\mathfrak{B}}(\R_I^{n})$, we may also assume that $\lambda _\infty  (P_{N_l}  \cap P_{N_m} ) = 0,\;l \ne m$.  We now have that
\[
\lambda _\infty  (P_N ) = \sum\limits_{k = 1}^L {\lambda _\infty  (P_{N_k } )}  \leqslant \sum\limits_{k = 1}^L {\lambda _\infty  (G_k )}  \leqslant \sum\limits_{k = 1}^\infty  {\lambda _\infty  (G_k )} .
\]
It follows that
\[
\lambda _\infty(\bigcup\nolimits_{k = 1}^{\infty} {G_k }) \le \sum\nolimits_{k = 1}^{\infty}\lambda _\infty({G_k }).
\] 
If the $G_k$ are disjoint, observe that if ${P} _N \subset{{P}'} _M$,
\[
\lambda _\infty  ({{P}'} _M ) \ge \lambda _\infty  ({P} _N) = \sum\limits_{k = 1}^L {\lambda _\infty  ({P} _{N_k } )}. 
\]
It follows that
\[
\lambda _\infty(\bigcup\nolimits_{k = 1}^{\infty} {G_k }) \ge \sum\nolimits_{k = 1}^{L}\lambda _\infty({G_k }).
\]
This is true for all $L$ so that this, combined with (3), gives our result. 
\end{proof}
If $F$ is an arbitrary compact set in ${\mathfrak{B}}(\R_I^{\infty})$, we define
\beqn
\lambda _\infty  (F) =  \inf \left\{ {\lambda _\infty  (G ):\,F\subset G,\;  G \; \rm{open} } \right\}.
\eeqn
\begin{rem}
At this point  we see the power of $\mathfrak{B}(\mathbb{R}_I^{\infty})$.  Unlike $\mathfrak{B}(\mathbb{R}^{\infty})$, equation (2.2) is well-defined for $\mathfrak{B}(\mathbb{R}_I^{\infty})$ because it has a sufficient number of open sets of finite measure.   
\end{rem}
\begin{thm} Equation (2.2) is consistent with Definition 2.6 and the results of Theorem 2.11.
\end{thm}
\begin{Def} Let $A$ be an arbitrary set in $\mathbb{R}_I^{\infty}$.  
\begin{enumerate}
\item The outer measure (on $\mathbb{R}_I^{\infty}$) is defined by:
\[  
\lambda _\infty^{*} (A) =  \inf \left\{ {\lambda _\infty  (G ):\,A\subset G,\;  G \; \rm{open} } \right\}.
\]
We let ${\mathfrak{L}}_0$ be the class of all $A$ with $\lambda _\infty^{*} (A) < \infty$.
\item If $A \in {\mathfrak{L}}_0$, we define the  inner measure of $A$ by
\[  
\lambda _{\infty,(*)} (A) =  \sup \left\{ {\lambda _\infty  (F ):\,F\subset A,\;  F \; \rm{compact} } \right\}.
\]
\item We say that $A$ is a bounded measurable set if $\lambda _{\infty}^* (A)=\lambda _{\infty,(*)} (A)$, and define the measure of $A,\; \lambda _{\infty}(A)$, by $\lambda _{\infty}(A)=\lambda _{\infty}^* A)$. 
\end{enumerate}  
\end{Def} 
\begin{thm} Let $A,\, B$ and $\{A_k\}$ be arbitrary sets in $\mathbb{R}_I^{\infty}$ with finite outer measure.
\begin{enumerate}
\item $\lambda _{\infty,(*)}(A) \le \lambda _{\infty}^*(A)$.
\item If $A \subset B$ then $\lambda _{\infty}^*(A) \le \lambda _{\infty}^*(B)$ and $\lambda_{\infty,(*)}(A) \le \lambda _{\infty,(*)}(B)$.
\item  
$
\lambda _{\infty}^*(\bigcup\nolimits_{k = 1}^{\infty} {A_k }) \le \sum\nolimits_{k = 1}^{\infty}\lambda _{\infty}^*({A_k }).
$
\item  If the $\{A_k\}$ are disjoint,  
$
\lambda_{\infty,(*)}(\bigcup\nolimits_{k = 1}^{\infty} {A_k }) \ge \sum\nolimits_{k = 1}^{\infty}\lambda_{\infty,(*)}({A_k }).
$
\end{enumerate}
\end{thm}
\begin{proof}  The proofs of (1) and (2) are straightforward.  To prove (3), let $\e>0$ be given.  Then, for each $k$, there exists an open set $G_k$ such that $A_k \subset G_k$ and $\lambda _{\infty}({G_k }) < \lambda _{\infty}^*({A_k }) +\e2^{-k}$.  Since $(\bigcup\nolimits_{k = 1}^{\infty} {A_k }) \subset (\bigcup\nolimits_{k = 1}^{\infty} {G_k })$, we have 
\[
\begin{gathered}
  \lambda _\infty ^* \left( {\bigcup\nolimits_{k = 1}^\infty  {A_k } } \right) \leqslant \lambda _\infty ^{} \left( {\bigcup\nolimits_{k = 1}^\infty  {G_k } } \right) \leqslant \sum\nolimits_{k = 1}^\infty  {\lambda _\infty ^{} (G_k )}  \hfill \\
  \quad \quad \quad \quad \quad  < \sum\nolimits_{k = 1}^\infty  {[\lambda _\infty ^* (A_k ) + \varepsilon } 2^{ - k} ] = \sum\nolimits_{k = 1}^\infty  {\lambda _\infty ^* (A_k ) + \varepsilon }. \hfill \\ 
\end{gathered} 
\]
Since $\e$ is arbitrary, we are done.

To prove (4), let $F_1,\,F_2, \dots,\,F_N$ be compact subsets of $A_1,\,A_2, \dots,\,A_N$, respectively.  Since the ${A_k}$ are disjoint, 
\[
\lambda _{\infty,(*)} \left( {\bigcup\nolimits_{k = 1}^\infty  {A_k } } \right) \geqslant \lambda _\infty ^{} \left( {\bigcup\nolimits_{k = 1}^N  {F_k } } \right) = \sum\nolimits_{k = 1}^N  {\lambda _\infty ^{} (F_k )}.   
\]
Thus,
\[
\lambda _{\infty,(*)} \left( {\bigcup\nolimits_{k = 1}^\infty  {A_k } } \right) \ge \sum\nolimits_{k = 1}^N  \lambda _{\infty,(*)} (A_k ).
\]
Since $N$ is arbitrary, we are done.
\end{proof}
The next two important theorems follow from the last one.
\begin{thm}(Regularity)  If $A$ has finite measure, then for every $\e >0$ there exist a compact set $F$ and an open set $G$ such that $F \subset A \subset G$, with $\lambda _\infty (G \setminus F ) < \e$.
\end{thm}
\begin{proof} Let $\e>0$ be given. Since $A$ has finite measure, it follows from our definitions of $\lambda _{\infty,(*)}$ and $\lambda _{\infty}^{*}$ that there is a compact set $F \subset A$ and an open set $G \supset A$ such that 
\[
\lambda_{\iy} (G) < \lambda_{\iy}^* (A) + \tfrac{ \in }
{2}\quad {\text{and}}\quad \lambda_{\iy} (F) > \lambda _{\iy, (*)} (A) - \tfrac{ \in }
{2}.
\]
Since $\lambda_{\iy} (G) = \lambda_{\iy} (F) + \lambda_{\iy} (G \setminus F)$, we have:
\[
\lambda_{\iy} (G \setminus F) = \lambda_{\iy} (G) - \lambda_{\iy} (F) < (\lambda_{\iy} (A) + \tfrac{\varepsilon }
{2}) - (\lambda_{\iy} (A) - \tfrac{\varepsilon }
{2}) = \varepsilon. 
\]
\end{proof}
\begin{thm}(Countable Additivity)  If the family $\{A_k\}$ consists of disjoint sets with bounded measure and $A={\bigcup\nolimits_{k = 1}^{\infty} {A_k } }$, with $\lambda_{\iy} ^* (A) <\iy$. then $\lambda _\infty (A ) = \sum\nolimits_{k = 1}^{\infty} \lambda _{\infty} (A_k )$.
\end{thm}
\begin{proof} Since $\lambda_{\iy} ^* (A)< \iy$, we have:
\[
  \lambda _\infty ^* (A) \leqslant \sum\nolimits_{k = 1}^\infty  {\lambda _\infty ^* (A_k )}    = \sum\nolimits_{k = 1}^\infty  {\lambda _{\infty ,(*)} (A_k )}  \leqslant \lambda _{\infty ,(*)} (A) \leqslant \lambda _\infty ^* (A). 
\]
It follows that   $\lambda _\infty (A)=\lambda _\infty ^* (A)= {\lambda _{\infty ,(*)} (A )}$, so that 
\[
\lambda _\infty  (A) = \lambda _\infty  \left( {\bigcup\nolimits_{k = 1}^\infty  {A_k } } \right) = \sum\nolimits_{k = 1}^\infty  {\lambda _\infty  (A_k )} .
\]
\end{proof}
\begin{Def} Let $A$ be an arbitrary set in $\mathbb{R}_I^{\infty}$.  We say that  $A$ is measurable if $A \cap M \in {\mathfrak{L}}_0$ for all  $M \in {\mathfrak{L}}_0$.  In this case, we define $\lambda _\infty(A)$ by:  
\[  
\lambda _\infty(A) =  \sup \left\{ {\lambda _\infty (A \cap M ):\,M \subset {\mathfrak{L}}_0 } \right\}.
\]
We let ${\mathfrak{L}}_I^\iy$ be the class of all measurable sets $A$. 
\end{Def}
Proofs of the following results are standard (see Jones \cite{J}, pages 48-52).
\begin{thm} Let $A$ and $\{A_k\}$ be arbitrary sets in ${\mathfrak{L}}_I^\iy$.
\begin{enumerate} 
\item If $\lambda _{\infty}^{*}(A) < \infty$, then $A \in {\mathfrak{L}}_0$ if and only if $A \in {\mathfrak{L}}_I^\iy$.  In this case, $\lambda_\infty(A)=\lambda _{\infty}^{*}(A)$.
\item ${\mathfrak{L}}_I^\iy$ is closed under countable unions,  countable intersections, differences and complements.
\item
\[
\lambda _{\infty}(\bigcup\nolimits_{k = 1}^{\infty} {A_k }) \le \sum\nolimits_{k = 1}^{\infty}\lambda _{\infty}({A_k }).
\]
\item  If $\{A_k\}$ are disjoint,  
\[
\lambda_{\infty}(\bigcup\nolimits_{k = 1}^{\infty} {A_k })= \sum\nolimits_{k = 1}^{\infty}\lambda_{\infty}({A_k }).
\]
\item If $A_k \subset A_{k+1}$ for all $k$,  then
\[
\lambda_{\infty}(\bigcup\nolimits_{k = 1}^{\infty} {A_k }) = \mathop {\lim }\limits_{k \to \infty } \lambda_{\infty}({A_k }).
\]
\item  If $A_{k+1} \subset A_k$ for all $k$ and $\lambda_{\infty}({A_1 }) < \iy$,  then
\[
\lambda_{\infty}(\bigcap\nolimits_{k = 1}^{\infty} {A_k }) = \mathop {\lim }\limits_{k \to \infty } \lambda_{\infty}({A_k }).
\]
\end{enumerate}
\end{thm}
We end this section with an important result that relates Borel sets to ${\mathfrak L}_I^\iy$-measurable sets (Lebesgue). 
\begin{thm} Let $A$ be a  ${\mathfrak L}_I^\iy$-measurable set.  Then there exists a Borel set $F$ and a set $N$ with $\lambda_{\infty}({N })=0$ such that $A=F \cup N$.
\end{thm} 
Thus, we see that $\lambda _{\infty} (\cdot)$ is a regular countably additive $\s$-finite Borel measure on $\mathbb{R}_I^{\infty}=\mathbb{R}^{\infty}$ (as sets).  More important is the fact that the development is no more difficult than the corresponding theory for Lebesgue measure on $\mathbb{R}^{n}$. 

Throughout the remainder of the paper we will also use $\mathfrak{B}[\mathbb{R}_I^{\infty}]$ for its completion ${\mathfrak L}_I^\iy$ when convenient.  This should cause no confusion since the given context will always be clear.  
\subsection{{Separable Banach Spaces}} 
In order to see what other advantages our construction of $(\mathbb{R}_I^{\infty}, \mathfrak{B}[\mathbb{R}_I^{\infty}], \lambda _{\infty} (\cdot))$ offers, in this section we study separable Banach spaces.   Let $\mcB$ be any separable Banach space. 

Recall that (see Diestel \cite{DI}, page 32):
\begin{Def}
A sequence $(u_n)$ is called a Schauder basis for $\mathcal{B}$ if $\lt\|u_n\rt\|_{\mcB}=1$ and, for each $f \in \mathcal{B}$, there is a unique sequence $(a_n)$ of scalars such that 
\[
f={\rm{lim}}_{n\rightarrow \infty}\sum\nolimits_{k = 1}^n {a_k u_k }. 
\] 
\end{Def}
\begin{Def}
A sequence $(v_n)$ is called an absolutely convergent Schauder basis for $\mathcal{B}$ if 
$\sum _{n = 1}^\infty  \left\| {v_n } \right\|_{\mcB}< \iy$ and, for each $f \in \mathcal{B}$, there is a unique sequence $(b_n)$ of scalars such that 
\[
f={\rm{lim}}_{n\rightarrow \infty}\sum\nolimits_{k = 1}^n {b_k v_k }. 
\] 
\end{Def}
\begin{lem}  Let $(u_n)$ be a Schauder basis for $\mathcal{B}$, then there exists an absolutely convergent Schauder basis for $\mathcal{B}$.
\end{lem}
\begin{proof}  Let $(v_n)=(\tf{u_n}{2^n})$.  Then
\[
\sum\limits_{n = 1}^\infty  {\left\| {v_n } \right\|_{\mcB}}  = \sum\limits_{n = 1}^\infty  {\frac{{\left\| {u_n } \right\|_{\mcB}}}
{{2^n }}}  = \sum\limits_{n = 1}^\infty  {\frac{1}
{{2^n }}}  = 1 < \infty .
\]
To see that $(v_n)$ is a Schauder basis for $\mathcal{B}$, let $f \in \mcB$.  By definition, there is a unique sequence $(a_n)$ of scalars such that 
\[
f={\rm{lim}}_{n\rightarrow \infty}\sum\nolimits_{k = 1}^n {a_k u_k }. 
\] 
If we take the sequence $(b_n)=(2^n a_n)$, then
\[
{\rm{lim}}_{n\rightarrow \infty}\sum\nolimits_{k = 1}^n {b_k v_k }={\rm{lim}}_{n\rightarrow \infty}\sum\nolimits_{k = 1}^n {a_k u_k }=f. 
\] 
\end{proof}
It is known that most of the natural separable Banach spaces, and all that have any use for applications in analysis, have a Schauder basis.  In particular, it is easy to see from  the definition of a Schauder basis that, for any sequence $(a_n) \in \R_I^{\iy}$ representing a function $f \in \mathcal{B}$, we have $\lim _{n \to \infty } a_n  = 0$.  It follows that every separable Banach space (with a Schauder basis) is isomorphic to a subspace of $\mathbb{R}_I^{\infty}$.

Let ${\mcB}_I$ be the set of all sequences $(a_n)$ for which ${\rm{lim}}_{n\rightarrow \infty}\sum\nolimits_{k = 1}^n {a_k u_k }$ exists in $\mathcal{B}$.  Define
\[
{\left\| {(a_n )} \right\|}_{{\mcB}_I} = \sup _n \left\| {\sum\nolimits_{k = 1}^n {a_k u_k } } \right\|_\mathcal{B}. 
\]
\begin{lem} An  operator
$$
T: ( {\mcB}, ||\cdot ||_{{\mcB}} )  \to ({\mcB}_I , ||\cdot ||_{{\mcB}_I}) ,
$$  
 defined  by  $T(f)=(a_k)$ for  $f=\lim_{n \to \infty}\sum_{k=1}^n  a_k  u_k   \in  {\mcB}$,
is an isomorphism from  ${\mcB}$ onto ${\mcB}_I$.
\end{lem}
Let ${\mcB}$ be a separable Banach space with a Schauder basis and let $\mcB_I= T[\mcB]$.  If $\mathfrak{B}({\mcB}_I)=\mcB_I \cap \mathfrak{B}[\R_I^\iy]$, we define the $\s-$algebra generated on $\mcB$, and associated with $\mathfrak{B}({\mcB}_I)$ by: 
\[
\mathfrak{B}_I [\mcB] = \left\{ {T^{ - 1} (A)\left| {\; A \in \mathfrak{B}[\mcB_I ]} \right.} \right\} = :T^{ - 1} \left\{ {\mathfrak{B}\left[ {\mcB_I } \right]} \right\}.
\]
Note that, just as $\mfB[\R^\iy] \subset \mfB[\R_I^\iy]$, we also have $\mathfrak{B} [\mcB] \subset \mathfrak{B}_I [\mcB]$ (with the containment proper).  
\begin{thm}Let $A \in \mathfrak{B}_I({\mcB})$ and  set ${\hat{\la}}_{{\mcB}}(A)= \la_{\infty}[T(A)]$.  Let $\la_{{\mcB}}$ be the completion of ${\hat{\la}}_{{\mcB}}$,  then ${\la}_{{\mcB}}$  is a non-zero $\sigma$-finite Borel measure on $\mcB$. 
\end{thm}
\begin{proof}  Let $\{ v_k\}$ be an absolutely convergent Schauder basis. We first prove that, for any $L>0$ and any sequence $(a_k) \in [-L,L]^{\aleph _o}$, the function $f= \sum _{k = 1}^\infty  a_k v_k \in {\mcB}$.  We then prove that  $\la_{{\mcB}}$ is nonzero.

Part 1

Let $L$ be given.  Since $(v_n)$ is an absolutely convergent Schauder basis, given $\e>0$ we can choose $N$ such that $\sum_{k=N}^{\infty}||v_k||< \frac{\epsilon}{L}$.  It follows that, for $N \le m \le n$, we have
\[
\left\| {\sum _{k = m}^n a_k v_k } \right\| \leqslant \sum _{k = m}^n \left\| {v_k } \right\| < \varepsilon. 
\]
Thus, the sequence $\{f_n\}$, defined by $f_n  = \sum _{k = 1}^n a_k v_k $, is a Cauchy sequence in ${{\mcB}}$.  Since ${{\mcB}}$ is a Banach space, the sequence converges.

Part 2

To prove that $\la_{{\mcB}}$ is nonzero, it suffices to show that $\la_{{\mcB}}\lt[T^{-1}\lt(I_0\rt)\rt] \ne 0$, where $(I_0)=[-\tf{1}{2},\tf{1}{2}]^{\aleph _o}$.  First, we note that $T$ is an injective linear map into $\R_I^{\iy}$, so that $B=T^{-1}(I_0) \in {\mathfrak B}_I({{\mcB}})$.  Thus,
\[
\lambda _{{\mcB}} (B) = \lambda _\infty  \left[ {T \left( { T^{ - 1} (I_0)} \right)} \right] = \lambda _\infty  (I_0) = 1.
\]
\end{proof}
\subsection{Translations}
In the theorem below, we will provide a new proof that $\ell_1$ is the largest (dense) group of admissible translations for $\R_I^\iy$, so necessarily $\ell_1$ is the largest group of admissible translations for every separable Banach space $\mcB$.

Recall that $h(x)=\otimes_{k=1}^{\iy}h_k(x_k)$, where $h_k(x_k)=1$, for $x_k \in [-\tf{1}{2}, \tf{1}{2}]$.  It follows from $d\nu =hd\la_\iy$, that $\nu$ is absolutely continuous with respect to  $\la_\iy$.  Thus,  $\nu$ is equivalent to $\la_{\iy}$.  Let $\mfT_{\la_{\iy}}$ be the set of admissible translations for $\R_I^\iy$ (i.e., $\la_\iy[A-x]= \la_\iy[A]$ for all $A \in  \mathfrak{B}[\R_I^\iy]$ and $x \in \mfT_{\la_{\iy}}$).
\begin{thm} If $A \in  \mathfrak{B}[\R_I^\iy]$ then $\la_\iy[A-x]= \la_\iy[A]$ if and only if $ \mathfrak{T}_{\la_{\iy}}=\ell_1$. 
\end{thm}
\begin{proof}
Suppose that $x \in \ell_1$.  Since $\nu  \sim \lambda _\infty$, we have that $ \mathfrak{T}_{\nu}=\mathfrak{T}_{\la_{\iy}}$ (see Yamasaki \cite{YA1}).  Thus, it suffices to prove that $\nu[A-x]= \nu[A]$.  By Kakutani's Theorem (\cite{KA}, see also \cite{HHK} pg. 116), $\nu[A-x] \sim \nu[A]$ if and only if 
\beqn
\prod\limits_{k = 1}^\infty  {\int_{ - \infty }^\infty  {\sqrt {h_k \left( {y_k } \right)h_k \left( {y_k  - x_k } \right)} d\lambda (y_k )} }  >0.
\eeqn
Now, 
\[
\int_{ - \infty }^\infty  {\sqrt {h_k \left( {y_k } \right)h_k \left( {y_k  - x_k } \right)} d\lambda (y_k ) = \int_{[ - \tfrac{1}
{2},\tfrac{1}
{2}] \cap [ - \tfrac{1}
{2} + x_k ,\tfrac{1}
{2} + x_k ]} {d\lambda (y_k )} }  = (1 - \left| {x_k } \right|)_{+},
\]
where $r_{+}=max(0,r)$.  Since $x \in \ell_1, \; \prod\nolimits_{k = n}^\infty  {(1 - \left| {x_k } \right|)_+}  >0$ for $n$ large enough.  Thus, equation (2.3)  will be satisfied for every $x \in \ell_1$,  so that $\ell_1 \subset \mathfrak{T}_{\nu}$. 

Now, suppose that $x \in \mathfrak{T}_{\la_{\iy}}$, so that $\la_\iy[A-x]= \la_\iy[A]$ for all $A \in  \mathfrak{B}[\R_I^\iy]$.  Thus, for $A \in \mathfrak{B}[\R_I^n]$, we have 
\[
\begin{gathered}
  \lambda _\infty  \left[ {A - x} \right] = \lambda _n \left[ {A_n  - x_n } \right] \cdot \prod\limits_{k = n + 1}^\infty  {\lambda \left\{ {\left[ { - \tfrac{1}
{2},\tfrac{1}
{2}} \right] \cap \left[ { - \tfrac{1}
{2} - x_k ,\tfrac{1}
{2} - x_k } \right]} \right\}}  \hfill \\
  \quad \quad \quad \quad  = \lambda _n \left[ {A_n } \right] \cdot \prod\limits_{k = n + 1}^\infty  {\lambda \left\{ {\left[ { - \tfrac{1}
{2},\tfrac{1}
{2}} \right] \cap \left[ { - \tfrac{1}
{2} - x_k ,\tfrac{1}
{2} - x_k } \right]} \right\}}  = \lambda _n \left[ {A_n } \right] \cdot \prod\limits_{k = n + 1}^\infty  {\left( {1 - \left| {x_k } \right|} \right)_ +  } . \hfill \\ 
\end{gathered} 
\]
If $A_n=I_n=\times_{k=1}^{n}[-\tf{1}{2}, \tf{1}{2}]$, we have $
1 = \mathop {\lim }\limits_{n \to \infty } \prod\limits_{k = n + 1}^\infty  {\left( {1 - \left| {x_k } \right|} \right)_ +  } 
$.   It follows that $\sum\nolimits_{k = 1}^\infty  {\left| {x_k } \right| < \infty } 
$, so that  $x \in \ell_1$.
\end{proof}
In closing,  we note that, since $\la_\iy$ is complete and regular, it is metrically transitivity with respect to $\R_0^\iy$.  It follows from Theorem 0.5 that  $\la_\iy$ is unique (this comment also applies to $\la_{\mcB}$).   
\subsection{Gaussian measure}
If we replace Lebesgue measure by the infinite product Gaussian measure, $\mu_{\infty}$, on $\mathbb{R}^{\infty}$, we get countable additivity but lose rotational invariance. Furthermore, the $\mu_{\infty}$ measure of $l_2$ is zero.  On the other hand, another approach is to use the standard projection method onto finite dimensional subspaces to construct a  probability measure directly on $l_2$. In this case, we recover rotational invariance but not translation invariance (and lose countable additivity).  The resolution of this problem led to the development of the Wiener measure \cite{WSRM} and this is where we are today. A nice discussion of this and related issues can be found in Dunford and Schwartz \cite{DS} (see pg. 402). 

We now turn to take a look at infinite product Gaussian measure from our new perspective. The canonical Gaussian measure on $\R$ is defined by:
\[
d\mu (x) = \frac{1}
{{\sqrt {2\pi } }}\exp \left\{ { - \frac{{\left| x \right|^2}}
{2}} \right\}d\lambda (x).
 \]
Recall that $\mu _\infty   =  \otimes _{k = 1}^\infty \mu$ is countably additive on $\mathbb{R}^{\infty}$, but its measure of  $\ell_2$ is zero.
If we introduce a scaled version of Gaussian measure on $\mathbb{R}_{I}^{\infty}$, we can resolve this difficulty.    We seek a family of variances $\{ \sigma_k^2 \}$  such that 
$$\mu_{{\mcB}}({\mcB})=\mathop  \otimes \limits_{k = 1}^\infty \mu_k(T[{\mcB}]=1,$$
where  $\mu_k$ is a linear Gaussian measure on $R$ with parameters $(0, \sigma_k)$  for $k \in N$ and  $\mu_{{\mcB}}$  is  defined by: 
\[
\mu_{{\mcB}}(B)=\mathop  \otimes \limits_{k = 1}^\infty  \mu _k (T[B]),
\]
for   any     Borel  subset  $B$ of  ${\mcB}$.
\begin{lem}Let $\left\{ \sigma _k^2 \right\}$ be a family of variances such that
\[
\sum\limits_{k = 1}^\infty{{\sigma _k^2 }}< \iy,
\]
then $\mu _{{\mcB} } \lt({T^{-1}([-L, L]^{\aleph_o })} \rt)>0$ for every positive number $L $.
\end{lem}
\begin{proof} Let $\{X_k\}$ be the family of independent Gaussian random variables defined on some common probability space, $\lt(\Om, \ \mathfrak{B}, \ P[\cdot]   \rt)$, with law $\mu_k$. 
If $X=(X_1,X_2, \dots)$, then
\[
\begin{gathered}
  P\left[ {\left\{ {\omega  \in \Omega \left| {} \right.X(\omega ) \in [ - L,L]^{\aleph _o } } \right\}} \right] = P\left[ {\bigcap\nolimits_{k = 1}^\infty  {\left\{ {\omega  \in \Omega \left| {} \right.X_k (\omega ) \in [ - L,L]} \right\}} } \right] \hfill \\
   = \prod\nolimits_{k = 1}^\infty  {P\left[ {\left\{ {\omega  \in \Omega \left| {} \right.\left| {X_k (\omega )} \right| \leqslant L]} \right\}} \right]}  \geqslant \prod\nolimits_{k = 1}^\infty  {\left( {1 - \frac{{\sigma _k^2 }}
{{L^2 }}} \right)} ,\quad {\text{by Chebyshev's inequality}}. \hfill \\ 
\end{gathered} 
\]
Clearly the product is positive.  We are done since $B=T^{-1}([-L, L]^{\aleph_o }) \in {\mathfrak{B}}(\mcB)$ and
\[
\mu _{{\mcB} } \lt( B \rt) = \left( {\otimes _{k = 1}^\infty  \mu _k } \right)\left( {T[T^{-1}([-L, L]^{\aleph_o })]} \right) = P\left[ {\left\{ {\omega  \in \Omega \left| {} \right.X(\omega ) \in [ - L,L]^{\aleph _o } } \right\}} \right].
\]
\end{proof}
\begin{thm} If the family of variances $\left\{ \sigma _k^2 \right\}$ satisfies the stronger condition
\beqn
\sum\limits_{k = 1}^\infty  {\frac{{\sigma _k^2 }}
{{\lt|x_k\rt| }}}  < \infty 
\eeqn
for some sequence $(x_k )\in  \ell_1$,
then $\mu _{{\mcB}}([{\mcB}])=1$.
\end{thm}
\begin{proof}  By definition, if $f \in \mcB$ and $(u_n)$ is a Schauder basis for $\mcB$, then there is a sequence of scalars $(a_k)$ such that $f=\lim_{n \to \iy}\sum_{k=1}^n{a_k u_k}$.  Since $T(f)=(a_k)$,
\[
\left| {\left\| {(a_n )} \right\|} \right|_{{\mcB}_I}  = \mathop {\sup }\limits_n \left\| {\sum\limits_{k = 1}^n {a_k u_k } } \right\|_{\mcB} \leqslant \left[ {\sum\limits_{k = 1}^\infty  {\left| {a_k } \right| } } \right] ,
\]
so that, if $(a_n ) \in \ell_1$, then $(a_n ) \in T(\mcB) ={\mcB}_I$.

Suppose that there is a sequence $(x_k) \in \ell_1 $ such that such that the inequality (2.3) is satisfied.   As in Lemma
2.24, by Chebyshev's inequality and inequality (2.3) we have
\[
\mu _{\mcB} \left\{ {T^{ - 1} \left( {\mathop  \times \limits_{k = 1}^\infty  \left[ { - \left| {x_k } \right|^{1/2} ,\,\left| {x_k } \right|^{1/2} } \right]} \right)} \right\} > 0.
\]
If  $A_n=R^n \times (\times_{k=n+1}^{\infty}[-\lt|x_k\rt|^{1/2},\lt|x_k\rt|^{1/2}])$,  then $A_n \subset A_{n+1}$  and $A_n \subseteq \mcB_I$ for all natural $n$.  Thus, we have
\[
\begin{gathered}
\mu_{\mcB}[T^{-1}({\mcB}_I)] \ge \lim_{n \to \infty} \mu_{\mcB}[T^{-1}(A_n)]  \hfill \\ 
=\lim_{n \to \infty}\prod_{k=n+1}^{\infty}
\mu_k([-\lt|x_k\rt|^{1/2},\lt|x_k\rt|^{1/2}]) \ge \lim_{n \to \infty}\prod\limits_{k=n + 1}^\infty  {\left( {1 - \frac{{\sigma _k^2 }}
{{\lt|x_k\rt| }}} \right)}
=1. \hfill \\ 
\end{gathered}
\]
\end{proof} 
\begin{Def} We call $\mu _{{\mcB} }$ a scaled version of Gaussian measure for  $\mathcal{B}$.
\end{Def}  
\begin{thm}  The measure $\mu _{{\mcB}}$ is a countably additive version of Gaussian measure on $\mathcal{B}$. 
\end{thm}
In particular, observe that we obtain a countably additive version of  Gaussian measure for both $\ell_2 $ and ${\mathbb{C}}_0[0,1]$ (the continuous functions $x(t)$ on $[0,1]$ with  $x(0)=0$).
\subsection{Rotational  Invariance}
In this section we study rotational invariance on subspaces of $\lt( \R_I^{\iy}, \ {\mathfrak{B}}_{I}[\R^{\iy}] \ \la _{\iy} \rt)$.   First, we need a little more information about Gaussian measures on vector spaces. (See Yamasaki \cite{YA}, pg. 151, for a proof of the next Theorem).  

Let $\mcF$ be a a real vector space,  let ${\mcF}^a$ be its algebraic dual space, and let ${\mathfrak{B}}_{\mcF}$ be the smallest $\s$-algebra such that ${\mfL}(x)$ is measurable for each functional ${\mfL} \in {\mcF}^a$ and all $x \in \mcF$.
\begin{thm} If $\mu$ is a measure on $({\mcF}^a, \ {\mathfrak{B}}_{\mcF})$, then the following are equivalent.
\begin{enumerate}
\item The Fourier transform of $\mu, \; {\hat{\mu}}$, is of the form:
\[
{\hat{\mu}}(x) =\exp \left\{ { - \tfrac{1}
{2}\left\langle {x,x} \right\rangle } \right\},
\]
for some inner product on $\mcF$.
\item For every $x \in \mcF$, the distribution of ${\mfL}(x)$ is a one-dimensional Gaussian measure.
\end{enumerate}
\end{thm}
In this general setting, a measure $\mu$ is said to be Gaussian on $({\mcF}^a, \ {\mathfrak{B}}_{\mcF})$ if it satisfies either of the above conditions.
\begin{ex} Let $\mcF =\R_0^{\iy}$, the set of sequences that are zero except for a finite number of terms and let $\left\langle {\cdot, \ \cdot} \right\rangle$ be the inner product on $\R_0^{\iy}$.  It is easy to show that the corresponding  measure on ${\mcF}^a=\R^{\iy}$ (satisfying either (1) or (2) above) is the infinite product Gaussian measure. 
\end{ex}
To understand the importance of this example, let $(a_n)$ be any sequence of positive numbers and let 
\beqn
\mcH_a  = \left\{ {\left. {{\mathbf{x}} \in \mathbb{R}^\infty  } \right| \; \sum\limits_{n = 1}^\infty  {a_n^2 x_n^2  < \infty } } \right\}.
\eeqn
The proof of the following is due to Yamasaki (\cite{YA}, pg. 153).
\begin{lem} If $a \in \ell_2$,  $\mu[{\mcH_a }]=1$, and if $a \notin \ell_2$,  $\mu[{\mcH_a }]=0$.
\end{lem}

Now, let us note that the standard one-dimensional Gaussian density, which is normally written as $f_X(x)=[\sqrt{2\pi}]^{-1}exp\{-\tf{1}{2}\lt|x\rt|^2\}$, may also be written as $f_X(x)=exp\{-\pi \lt|x\rt|^2\}$ with no factors of $\sqrt{2\pi}$ if we scale $x \to \tf{x}{\sqrt{2\pi}}$.  With this convention, we can write the infinite dimensional version for $L^2[\mcH, \la_{\mcH}]$ as the derivative of the  Gaussian distribution $\mu _{{\mcH}}$ with respect to the Lebesgue measure on $\mcH$:
\beqn
f({\mathbf{x}}) = \exp \{  - \pi \left| {\mathbf{x}} \right|_\mcH ^2 \}  = \frac{{d\mu_{{\mcH}} ({\mathbf{x}})}}
{{d\lambda _{\mcH} ({\mathbf{x}})}}.
\eeqn 
This shows that, with the appropriate definition of Lebesgue measure, there is a corresponding density for a Gaussian distribution on Hilbert space.
\begin{rem} In the general case (see DePrato \cite{DP}), when $\bQ$ is a (positive definite) trace-class operator  and ${\bf x}$ is a Gaussian random variable with mean $\bf m$ and covariance $\bQ$, we can write equation (2.6) as:
\beqa
f({\mathbf{x}}) = \left[ {\det \mathbb{Q}} \right]^{ - 1/2} \exp \left\{ { - {\pi} \left\langle {\mathbb{Q}^{ - 1} ({\mathbf{x}} - {\mathbf{m}}),({\mathbf{x}} - {\mathbf{m}})} \right\rangle _\mcH } \right\}\frac{{d\mu _\mcH ({\mathbf{x}})}}
{{d\lambda _\mcH  ({\mathbf{x}})}}.
\eeqa 
\end{rem}
\begin{Def} A rotation on $\mcH$ is a bijective isometry $U: \mcH \to \mcH$.
\end{Def}
It is well-known that $\mu _{\mcH}$ is invariant under rotations over $\lt( \mcH, \ {\mathfrak{B}}_{\mcH} \rt)$ (see Yamasaki \cite{YA}, pg. 163).  
\begin{thm} The  measure, $\la _{\mcH}$, is invariant under rotations  and $\mathfrak{R}=:(T^{-1}(\ell_2))$ is  dense in ${\mcH}$ and the maximal rotation invariance subspace for $\la _{\mcH}$.
\end{thm}
\begin{proof} Let any measurable set $A \in {\mathfrak{B}}_{\mcH}$.  If $U$ is any rotation on $\mcH$, then $\mu _{\mcH}(UA)=\mu _{\mcH}(A)$ and $\left| {U\mathbf{x}} \right|_\mcH ^2=\left| {\mathbf{x}} \right|_\mcH ^2$.  It follows from equation (2.6) that $\la _{\mcH}(UA)=\la _{\mcH}({A})$.

It follows from $\R_0^\iy \subset \mfR \subset \mcH$, that $\mfR$ is dense, and from Lemma 2.33 that $\mfR$ is maximal.
\end{proof}
\subsection*{{Discussion}}
In this section, we have shown that what appears to be a minor change in the way we represent $\mathbb{R}^{\infty}$ makes it possible to define an analogue of both Lebesgue and Gaussian measure  (countably additive) on every (classical) separable Banach space with a Schauder basis.  Furthermore, our version of Gaussian measure is  rotationally invariant, a property not shared by Wiener measure.  (What is more important, we have obtained our core results using basic methods  of Lebesgue measure theory from $\R^n$.)
\section{\bf Operators}
This section provides the  background to understand the relationship between operators defined on ${\mathcal{H}}_ \otimes ^2$ (which is nonseparable), and their restriction to ${\mathcal{H}}_ \otimes ^2(h)$. We also obtain general conditions that allow us to define infinite sums and products of linear operators on  ${\mathcal{H}}_ \otimes ^2(h)$ for a given $h$.
\subsection{Bounded Operators on ${\mathcal{H}}_ \otimes^2$} 
	In this section we review the class of bounded operators on ${\mathcal{H}}_ \otimes ^2$ and their relationship to those on each ${\mathcal{H}}_i$.  Many of the results are originally due to von Neumann \cite{VN2}.  However, the proofs are new or simplified versions (some from the literature). 
	  		
	Let $L[{\mathcal{H}}_ \otimes ^2  ]$ be the set of bounded operators on ${\mathcal{H}}_ \otimes ^2$.  For each fixed $i _0  \in \N$ and $A_{i _0 }  \in L({\mathcal{H}}_{i _0 } )$, define ${\mcA}_{i _0}  \in L({\mathcal{H}}_ \otimes ^2  )$ by: 
\[
{\mcA}_{i _0}(\sum _{k = 1}^N  \otimes _{i  \in \N} g _i ^k ) = \sum _{k = 1}^N A_{i _0 } g _{i _0 }^k   \otimes ( \otimes _{i  \ne i _0 } g _i ^k )
\]
for $\sum _{k = 1}^N  \otimes _{i  \in \N} g _i ^k$ in ${\mathcal{H}}_ \otimes ^2$ and $N$ finite but arbitrary.  Extending to all of  ${\mathcal{H}}_ \otimes ^2$ produces an isometric isomorphism of  $L[{\mathcal{H}}_{i _0 } ]$ into $L[{\mathcal{H}}_ \otimes ^2  ]$, which we denote by $L[{\mathcal{H}}(i _0 )]$, so that the relationship $L[{\mathcal{H}}_i ] \leftrightarrow L[{\mathcal{H}}(i )]$ is an isometric isomorphism of algebras.  Let $L^\#  [{\mathcal{H}}_ \otimes ^2  ]$ be the uniform closure of the algebra generated by $\{ L[{\mathcal{H}}(i )],\;i  \in \N \}$.  It is clear that $L^\#  [{\mathcal{H}}_ \otimes ^2  ] \subset L[{\mathcal{H}}_ \otimes ^2  ]$.  von Neumann has shown that the inclusion becomes equality if and only if $\N$ is replaced by a finite set.  On the other hand, $L^\#  [{\mathcal{H}}_ \otimes ^2  ]$ clearly consists of all operators on ${\mathcal{H}}_ \otimes ^2$ that are generated directly from the family $\{ L[{\mathcal{H}}(i )],\;i  \in \N \}$ by algebraic and topological processes.   

Let ${\mathbf{P}}_g ^s$ denote the projection from ${\mathcal{H}}_ \otimes ^2$ onto ${\mathcal{H}}_ \otimes ^2  (g )^s$, and let ${\mathbf{P}}_g ^w$ denote the projection from ${\mathcal{H}}_ \otimes ^2$ onto ${\mathcal{H}}_ \otimes ^2 (g )^w$. 
\begin{thm}
 If ${\mathbf{T}} \in L^\#  ({\mathcal{H}}_ \otimes ^2  )$, then ${\mathbf{P}}_g ^s {\mathbf{T}} = {\mathbf{TP}}_g ^s$ and ${\mathbf{P}}_g ^w {\mathbf{T}} = {\mathbf{TP}}_g ^w$.
\end{thm}
\begin{proof} 
The weak case follows from the strong case, so we prove that ${\mathbf{P}}_g ^s {\mathbf{T}} = {\mathbf{TP}}_g ^s$.   Since vectors of the form $G  = \sum\nolimits_{k = 1}^L { \otimes _{i  \in \N} g _i ^k }$, with $g _i ^k  = g _i$ for all but a finite number of $i$, are dense in ${\mathcal{H}}_ \otimes ^2 (g )^s$; it suffices to show that ${\mathbf{T}}f  \in {\mathcal{H}}_ \otimes ^2  (g )^s$.  Now, ${\mathbf{T}} \in L^\#  ({\mathcal{H}}_ \otimes ^2  )$ implies that there exists a sequence of operators ${\mathbf{T}}_n$ such that $\left\| {{\mathbf{T}} - {\mathbf{T}}_n } \right\|_ \otimes    \to 0$ as $n \to \infty$, where each ${\mathbf{T}}_n$ is of the form: ${\mathbf{T}}_n  = \sum\nolimits_{k = 1}^{N_n } {a_k^n T_k^n }$, with $a_k^n$  a complex scalar, $N_n  < \infty$, and each $T_k^n  = \hat  \otimes _{i  \in \M_k } T_{ki }^n \hat  \otimes _{i  \in \N \backslash \M_k } I_i$ for some finite set of  $i$-values $\M_k$, where $I_i$ is the identity operator on ${\mcH}_i$.  Hence,
\[
{\mathbf{T}}_n f  = \sum\nolimits_{l = 1}^L {\sum\nolimits_{k = 1}^{N_n } {a_k^n  \otimes _{i  \in \M_k } T_{k i }^n g _i ^l  \otimes _{i  \in \N \backslash \M_k } g _i ^l } }. 
\]
Now, it is easy to see that, for each $l$, $ \otimes _{i  \in \M_k } T_{k i }^n g _i ^l  \otimes _{i  \in \N \backslash \M_k } g _i ^l  \equiv ^s  \otimes _{i  \in \N} g _i$.  It follows that ${\mathbf{T}}_n f  \in {\mathcal{H}}_ \otimes ^2  (g )^s$ for each $n$, so that ${\mathbf{T}}_n  \in L[{\mathcal{H}}_ \otimes ^2  (g )^s ]$.  Since $L[{\mathcal{H}}_ \otimes ^2  (g )^s ]$ is a norm closed algebra, ${\mathbf{T}} \in L[{\mathcal{H}}_ \otimes ^2  (g)^s ]$ and it follows that ${\mathbf{P}}_g ^s {\mathbf{T}} = {\mathbf{TP}}_g ^s$. 
\end{proof}
	Let $z_i   \in {\mathbf{C}}$, $\;\left| {z_i  } \right| = 1$, and define $U[{\mathbf{z}}]$ by: $U[{\mathbf{z}}] \otimes _{i  \in \N} g _i   =  \otimes _{i  \in \N} z_i  g _i$.
\begin{thm} The operator $U[{\mathbf{z}}]$ has a unique extension to a unitary operator on ${\mathcal{H}}_ \otimes ^2$, which we also denote by $U[{\mathbf{z}}]$, such that:
\begin{enumerate}
\item $U[{\mathbf{z}}]:\;\,{\mathcal{H}}_ \otimes ^2  (g )^w  \to {\mathcal{H}}_ \otimes ^2  (g )^w$, so that ${\mathbf{P}}_g ^w U[{\mathbf{z}}] = U[{\mathbf{z}}]{\mathbf{P}}_g ^w$.  
\item If $\prod _\nu  z_\nu$  is quasi-convergent but not convergent, then $U[{\mathbf{z}}]:\;\,{\mathcal{H}}_ \otimes ^2  (g )^s  \to {\mathcal{H}}_ \otimes ^2  (h )^s$, for some $h  \in \,{\mathcal{H}}_ \otimes ^2  (g )^w$ with $g  \bot h$.
\item $U[{\mathbf{z}}]:\;\,{\mathcal{H}}_ \otimes ^2  (g )^s  \to {\mathcal{H}}_ \otimes ^2  (g )^s$ if and only if $\prod _i  z_i$ converges  and $U[{\mathbf{z}}] = (\prod _i  z_i  ){\mathbf{I}}_ \otimes$, where ${\bf I}_\otimes$ is the identity operator on ${\mathcal{H}}_ \otimes ^2 $. This implies that ${\mathbf{P}}_g ^s U[{\mathbf{z}}] = U[{\mathbf{z}}]{\mathbf{P}}_g ^s$.
\end{enumerate}
\end{thm}
\begin{proof} For (1), let $h  = \sum _{k = 1}^N  \otimes _{i  \in \N} h _i ^k$, where $\otimes _{i  \in \N} h _i ^k  \equiv ^w  \otimes _{i  \in \N} g _i, \;N$ is arbitrary and $1 \leqslant k \leqslant N$.  Then
\[
U^ *  [{\mathbf{z}}]U[{\mathbf{z}}]h  = \sum _{k = 1}^N  \otimes _{i  \in \N} z_i ^ *  z_i  h _i ^k  = h  = U[{\mathbf{z}}]U^ *  [{\mathbf{z}}]h. 
\]
Thus, we see that $U[{\mathbf{z}}]$ is a unitary operator, and since $h$ of the above form are dense, $U[{\mathbf{z}}]$ extends to a unitary operator on ${\mathcal{H}}_ \otimes ^2$.  By definition, $\sum _{k = 1}^N  \otimes _{i  \in \N} z_i  h _i ^k  \in {\mathcal{H}}_ \otimes ^2  (g )^w$ if $\sum _{k = 1}^N  \otimes _{i  \in \N} h _\nu ^k  \in {\mathcal{H}}_ \otimes ^2  (g )^w$, so that $U[{\mathbf{z}}]:\;\,{\mathcal{H}}_ \otimes ^2  (g )^w  \to {\mathcal{H}}_ \otimes ^2  (g )^w$ and ${\mathbf{P}}_g ^w U[{\mathbf{z}}] = U[{\mathbf{z}}]{\mathbf{P}}_g ^w$.  
To prove (2), use Theorem 1.6 (3) and (4) to note that $\prod _i  z_i   = 0$ and $\otimes _{i  \in \N} h _i ^k  \equiv ^s  \otimes _{i  \in \N} g _i$ imply that $\otimes _{i  \in \N} z_i  h _i ^k  \in {\mathcal{H}}_ \otimes ^2  (f )^s$ with ${\mathcal{H}}_ \otimes ^2  (f )^s  \bot {\mathcal{H}}_ \otimes ^2  (g )^s$.
To prove (3), note that, if $0 < \left| {\prod _i  z_i  } \right| < \infty$, then $U[{\mathbf{z}}] = [(\prod _i  z_i  ){\mathbf{I}}_ \otimes  ]$, so that $U[{\mathbf{z}}]:\;\,{\mathcal{H}}_ \otimes ^2  (g )^s  \to {\mathcal{H}}_ \otimes ^2  (g )^s$.  Now suppose that $U[{\mathbf{z}}]:\;\,{\mathcal{H}}_ \otimes ^2  (g )^s  \to {\mathcal{H}}_ \otimes ^2  (g )^s$, then $\otimes _{i  \in \N} z_i  h _i ^k  \equiv ^s  \otimes _{i  \in \N} g _i$ and so $\prod _i  z_i$ must converge.  Therefore, $U[{\mathbf{z}}]h  = [(\prod _i  z_i  ){\mathbf{I}}_ \otimes  ]h$ and ${\mathbf{P}}_g ^s U[{\mathbf{z}}] = U[{\mathbf{z}}]{\mathbf{P}}_g ^s$. 

It is easy to see that, for each fixed ${i  \in \N}$, ${\mcA}(i ) \in L[{\mathcal{H}}(i )]$ commutes with any ${\mathbf{P}}_g ^s ,\;{\mathbf{P}}_g ^w$ or $U[{\mathbf{z}}]$, where $g$ and ${\mathbf{z}}$ are arbitrary.
\end{proof}
\begin{thm} Every ${\mathbf{T}} \in L^\#  [{\mathcal{H}}_ \otimes ^2  ]$ commutes with all ${\mathbf{P}}_g ^s ,\;{\mathbf{P}}_g ^w$ and $U[{\mathbf{z}}]$, where $g$ and ${\mathbf{z}}$ are arbitrary.
\end{thm}
\begin{proof} Let $\mathfrak{L}$ be the set of all ${\mathbf{P}}_g ^s ,\;{\mathbf{P}}_g ^w$ or $U[{\mathbf{z}}]$, with $g$ and ${\mathbf{z}}$ arbitrary.  From the above observation, we see that all ${\mcA}_i  \in L[{\mathcal{H}}(i )], \;{i  \in \N}$, commute with $\mathfrak{L}$ and hence belong to its commutator $\mathfrak{L}'$.   Since $\mathfrak{L}'$ is a closed algebra, this implies that $L^\#  [{\mathcal{H}}_ \otimes ^2  ] \subseteq \mathfrak{L}'$ so that all ${\mathbf{T}} \in L^\#  [{\mathcal{H}}_ \otimes ^2  ]$ commute with $\mathfrak{L}$.
\end{proof}
\subsection{Unbounded Operators on ${\mathcal{H}}_\otimes^2$} 
In this section, we consider a restricted class of unbounded operators and the notion of a strong convergence vector introduced by Reed \cite{RE}.

For each $i \in \N$, let $A_i$ be a closed densely defined linear operator on ${\mathcal{H}}_i$, with domain $D(A_i )$, and let ${\mcA}_i$ be its extension to ${\mathcal{H}}_ \otimes ^2$, with domain $D({\mcA}_i ) \supset \tilde D({\mcA}_i ) = D(A_i ) \otimes ( \otimes _{k \ne i} {\mathcal{H}}_k )$.  The next theorem follows directly from the definition of the tensor product of semigroups.
\begin{thm} Let $A_i ,\;1 \leqslant i \leqslant n$, be generators of a family of $C_0$-semigroups $S_i (t)$ on ${\mathcal{H}}_i$ with $\left\| {S_i (t)} \right\|_{{\mathcal{H}}_i }  \leqslant M_i e^{\omega _i t}$.  Then ${\mathbf{S}}_n (t) = \hat  \otimes _{i = 1,n}  {\kern 1pt} S_i (t)$, defined on $\hat  \otimes _{i = 1,n}  {\kern 1pt} {\mathcal{H}}_i$, has a unique extension (also denoted by ${\mathbf{S}}_n (t)$) to all of ${\mathcal{H}}_ \otimes ^2$, such that, for all vectors $\sum _{k = 1}^K  \otimes _{i \in \N} g _i^k$ with $g _l^k  \in D(A_l ),\;1 \leqslant l \leqslant n$, the infinitesimal generator for ${\mathbf{S}}_n (t)$ satisfies:
\[
{\mcA}^n \left[ {\sum _{k = 1}^K  \otimes _{i \in \N} g _i^k } \right] = \sum _{l = 1}^n \sum _{k = 1}^K A_l g _l^k ( \otimes _{i \in \N}^{i \ne l} g _i^k ).
\]		
\end{thm}
\begin{Def} Let $\{ A_i \}$, ${i \in \N}$, be a family of closed densely defined linear operators on ${\mathcal{H}}_i$ and let $g _i  \in D(A_i )$ (respectively, $f _i  \in D(A_i )$), with $\left\| {g _i } \right\|_{\mathcal{H}}  = 1$ (respectively, $\left\| {f _i } \right\|_{\mathcal{H}}  = 1$), for all ${i \in \N}$. 
\begin{enumerate}
\item We say that $g  =  \otimes _{i \in \N} g _i$ is a strong convergence sum (scs)-vector for the family $\{ {\mcA}_i \}$ if  $\mathop {s{\text{ - }}\lim }\limits_{n \to \infty } \sum\nolimits_{k = 1}^n {{\mcA}_k g }   = \sum _{k = 1}^{\infty} A_k g _k ( \otimes _{i \in \N}^{i \ne k} g _i )$ exists.  
\item We say that $f  =  \otimes _{i \in \N} f _i$ is a strong convergence product (scp)-vector for the family $\{ {\mcA}_i \}$ if $\mathop {s{\text{ - }}\lim }\limits_{n \to \infty } \prod\nolimits_{k = 1}^n {{\mcA}_k f }  =   \otimes _{i \in \N} A_i f _i $ exists. 
\end{enumerate}
\end{Def}
Let ${\mathbf{D}}_g$  be the linear span of $\{ \chi  =  \otimes _{i \in \N} \chi _i ,\;\chi _i  \in D(A_i )\}$, with $\;\chi _i  = g _i $ (and let ${\mathbf{D}}_f$  be the linear span of $\{ \eta  =  \otimes _{i \in \N} \eta _i ,\;\eta _i  \in D(A_i )\}$, with $\;\eta _i  = f _i $) for all $i > L$, where $L$ is arbitrary but finite.  Clearly, ${\mathbf{D}}_g$  is dense in ${\mathcal{H}}_ \otimes ^2  (g )^s$ (${\mathbf{D}}_\eta$  is dense in ${\mathcal{H}}_ \otimes ^2  (f )^s$).  If there is a possible chance for confusion, we let ${\mcA}_s$, respectively ${\mcA}_p$,  denote the closure of  $\sum\nolimits_{k = 1}^\infty  {{\mcA}_k }$ on ${\mathcal{H}}_ \otimes ^2  (g )^s$ (respectively $\prod\nolimits_{k = 1}^\infty  {{\mcA}_k }$ on ${\mathcal{H}}_ \otimes ^2  (f )^s$).  It follows that ${\mathcal{H}}_ \otimes ^2  (g )^s$ (respectively ${\mathcal{H}}_ \otimes ^2  (f )^s$) are natural spaces for the study of infinite sums or products of unbounded operators.  (The notion of a strong convergence sum vector first appeared in Reed [RE].)
\begin{Def} We call ${\mathcal{H}}_ \otimes ^2  (g )^s$ an ${\mathcal{RS}}$-space (respectively, ${\mathcal{H}}_ \otimes ^2  (f )^s$ an ${\mathcal{RP}}$-space ) for the family $\{ {\mcA}_i \}$.
\end{Def}
Let  $\{ U_k (t) \} $ be a set of unitary groups on $\{ {\mathcal{H}}_k \} $.  It is easy to see  that $U(t) = \hat  \otimes _{k = 1}^\infty  U_k (t)$  is a unitary group on ${\mathcal{H}}_ \otimes ^2$.  However, we know from Theorem 3.2 (2), that it need not be reduced on any partial tensor product subspace.  The following results are due to  Streit [ST] and Reed [RE], as indicated.   
\begin{thm} (Streit) Suppose $\{ {\mcA}_k \} $ is a set of selfadjoint linear operators on the space ${\mathcal{H}}_ \otimes ^2 (g )^s $, with corresponding unitary groups $\{ U_k (t)\}$.  If $U(t) = \hat  \otimes _{k = 1}^\infty  U_k (t)$, then $
{\mathbf{P}}_g ^s U(t) = U(t){\mathbf{P}}_g ^s$  (i.e., $U(t)$ is reduced on ${\mathcal{H}}_ \otimes ^2 (g )^s $) and $U(t)$ is a strongly continuous unitary group on ${\mathcal{H}}_ \otimes ^2 (g )^s $ if and only if, for each $c>0$, the following three conditions are satisfied:
\begin{enumerate}
\item 
$\sum\nolimits_{k = 1}^\infty  {\left| {\left\langle {{\mcA}_k E_k [ - c,c]g _k ,g _k } \right\rangle } \right|}  < \infty$,
\item 
$\sum\nolimits_{k = 1}^\infty  {\left| {\left\langle {{\mcA}_k^2 E_k [ - c,c]g _k ,g _k } \right\rangle } \right|}$,
\item 
$\sum\nolimits_{k = 1}^\infty  {\left| {\left\langle {(I_k  - E_k [ - c,c]g _k ,g _k } \right\rangle } \right|}  < \infty$,
\end{enumerate}
where $E_k [ - c,c]$ are the spectral projectors of  ${\mcA}_k $ and, in this case, $
U(t) = s- \lim _{n \to \infty } \hat  \otimes _{k = 1}^n U_k (t)$.
\end{thm}
\begin{cor} Conditions 1-3 are satisfied if and only if there exists a strong convergence vector $g  =  \otimes _{k = 1}^\infty  g _k$ for the family $\{ A_k \} $ such that $g _k  \in D(A_k )$ and 
\[
\sum\nolimits_{k = 1}^\infty  {\left| {\left\langle {{\mcA}_k g _k ,g _k } \right\rangle } \right|}  < \infty ,{\text{    }}\sum\nolimits_{k = 1}^\infty  {\left\| {{\mcA}_k g _k } \right\|} ^2  < \infty. 
\]  
\end{cor}
\begin{thm} (Reed) $U(t)$ is reduced on ${\mathcal{H}}_ \otimes ^2 (g )^s$ and $U(t)$  is a strongly continuous unitary group on ${\mathcal{H}}_ \otimes ^2 (g )^s$  if and only if  $g  =  \otimes_{k = 1}^\infty  g$ is a strong convergence vector for the family $\{A_k \}$ and $\sum\nolimits_{k = 1}^\infty  {\left| {\left\langle {{\mcA}_k g _k ,g _k } \right\rangle } \right|}  < \infty$.  If each ${A}_k$ is positive, the statement is true without the above condition.  In either case, ${\mcA}$, the closure of $\sum\nolimits_{k = 1}^\infty  {{\mcA}_k } $, is the generator of $U(t)$. 
\end{thm}
The next result strengthens and extends Reed's theorem to contraction semigroups  (i.e., the positivity requirement above can be dropped). 
\begin{thm} Let $\{{S}_k (t)\}$ be a family of strongly continuous contraction semigroups with generators $\{A_k \}$  defined on $\{{\mathcal{H}}_k \} $, and let $
g  =  \otimes _{k = 1}^\infty  g _k$ be a strong convergence vector for the family $\{A_k \}$.   Then ${\bf{S}}(t) = \hat  \otimes _{k = 1}^\infty  S_k (t)$
 is reduced on ${\mathcal{H}}_ \otimes ^2  (g)^s$ and   is a strongly continuous contraction semigroup.  If ${\bf{S}}(t) = \hat  \otimes _{k = 1}^\infty  S_k (t)$ is reduced on ${\mathcal{H}}_ \otimes ^2  (g )^s$  and is a strongly continuous contraction semigroup on ${\mathcal{H}}_ \otimes ^2  (g )^s$, then there exists a strong convergence vector $f  =  \otimes _{k = 1}^\infty  f _k  \in \mathcal{H}_ \otimes ^2  (g )^s$ for the family $\{A_k \}$.
\end{thm}  
\begin{proof} 
Let $g  =  \otimes _{k = 1}^\infty  g _k$  be a strong convergence vector for the family $\{A_k \}$.   Without loss, we can assume that $\left\| {g _k } \right\| = 1$.  Let ${\mathbf{S}}_n (t) = \hat  \otimes _{k = 1}^n S_k (t)\hat  \otimes ( \otimes _{k = n + 1}^\infty  I_k )$ and observe that ${\bf{S}}_n(t)$ is a contraction semigroup on ${\mathcal{H}}_ \otimes ^2  (g )^s$  for all finite $n$ .  Furthermore, its generator is the closure of ${\mcA}^n  = \sum\nolimits_{k = 1}^n {{\mcA}_k } $, where ${\mcA}_k  = A_k \hat  \otimes ( \otimes _{i \ne k}^\infty  I_i )$.   If $n$ and $m$  are arbitrary, then 
\[
\begin{gathered}
  \left[ {{\mathbf{S}}_n (t) - {\mathbf{S}}_m (t)} \right]g  = \int_0^1 {} \frac{d}
{{d\lambda }}\left\{ {{\mathbf{S}}_n [\lambda t]{\mathbf{S}}_m [(1 - \lambda )t]} \right\}g d\lambda  \hfill \\
  {\text{                         }} = t\int_0^1 {} {\mathbf{S}}_n [\lambda t]{\mathbf{S}}_m [(1 - \lambda )t]\left[ {{\mcA}^n  - {\mcA}^m } \right]g d\lambda , \hfill \\ 
\end{gathered}
\]
where we have used the fact that, if two semigroups commute, then their corresponding generators also commute.  It follows that:
\[
\left\| {\left[ {{\mathbf{S}}_n (t) - {\mathbf{S}}_m (t)} \right]g } \right\| \leqslant t\left\| {\left[ {{\mcA}^n  - {\mcA}^m } \right]g } \right\|.
\]
Since $g  =  \otimes _{k = 1}^\infty  g$   is a strong convergence vector for the family $\{A_k \}$, it follows that \\ $s-\lim _{n \to \infty } {\mathbf{S}}_n (t) = {\mathbf{S}}(t)$ exists on a dense set in ${\mathcal{H}}_ \otimes ^2  (g )^s$ and the convergence is uniform on  bounded $t$ intervals.    It follows that $S(t)$ extends to a bounded linear operator on ${\mathcal{H}}_ \otimes ^2  (g )^s$.  To see that the closure of $S(t)$ must be a contraction, for any $\e>0$, choose $n$  so large that $\left\| {\left[ {{\mathbf{S}}_n (t) - {\mathbf{S}}(t)} \right]g } \right\|_ \otimes   < \varepsilon \left\| g  \right\|_ \otimes  $.  It follows that 
\[
\left\| {{\mathbf{S}}(t)g } \right\|_ \otimes   \leqslant \left\| {{\mathbf{S}}_n (t)g } \right\|_ \otimes   + \left\| {\left[ {{\mathbf{S}}_n (t) - {\mathbf{S}}(t)} \right]g } \right\|_ \otimes < \left\| g  \right\|_ \otimes  (1 + \varepsilon ).
\]
Thus, ${\bf{S}}(t)$ is a contraction operator on ${\mathcal{H}}_ \otimes ^2  (g )^s$.  It is easy to check that it is a $C_0$-semigroup.

Now suppose that ${\mathbf{S}}(t) = \hat  \otimes _{k = 1}^\infty  S_k (t)$ is a strongly continuous contraction semigroup which is reduced on ${\mathcal{H}}_ \otimes ^2  (g )^s$.   It follows that the generator $\mcA$ of ${\bf{S}}(t)$ is m-dissipative, and hence defined on a dense domain $D(\mcA)$  in ${\mathcal{H}}_ \otimes ^2  (g )^s$   with ${\bf{S'}}(t)f  = {\bf{S}}(t){\mcA}f  = \mcA{\bf{S}}(t)f $ for all $f \in D(\mcA)$. Since any such $f$ is of the form $
f  = \sum _{l = 1}^\infty   \otimes _{k = 1}^\infty  f _k^l $, each $f ^l  =  \otimes _{k = 1}^\infty  f _k^l$ is in $D(\mcA)$.  A simple computation shows that ${\mcA}f ^l  = \sum _{k = 1}^\infty  {\mcA}_k f ^l $, so that any $f ^l $ is a strong convergence vector for the family $\{ A_k \} $.
\end{proof}
It is easy to see that, in the second part of the theorem, we cannot require that $g  =  \otimes _{k = 1}^\infty  g_k$   itself be a strong convergence vector for the family $\{A_k \}$ since it need not be in the domain of $\mcA$.  For example,  $g_1  \notin D(A_1 )$, while $g_k  \in D(A_k ),\;k \ne 1$. 
\section{Function Spaces}
 Let $\chi_{I_n }$ be the indicator (or characteristic) function of $I_n=\times_{k=n+1}^{\iy}I$.  If we let $\mathfrak{L}(\mathbb{R}^n)$ represent the class of measurable functions on  $\mathbb{R}^n$, then for each measurable function $f_n \in \mathfrak{L}(\mathbb{R}^n)$ we identify $f \in \mathfrak{L}(\mathbb{R}_I^n)$ by $f=f_n \otimes \chi_{I_n }$.
\begin{Def} A real-valued function $f$ defined on the measure space $\lt(\mathbb{R}_I^{\iy}, {\mathfrak{B}}[\mathbb{R}_I^{\iy}], \la_{\iy}\rt)$ is said to be measurable if $f^{-1}(A) \in {\mathfrak{B}}[\mathbb{R}_I^{\iy}]$ for every $A \in {\mathfrak{B}}[\mathbb{R}]$.
\end{Def}
In this section we develop those aspects of function space theory that will be of use later.  We note that all the standard theorems for Lebesgue measure apply.  (The proofs are the same as for integration on $\R^n$.)  
\subsection{$L^1$-Theory}
Let ${{L}}^1 [\R_I^{n}]$ be the class of integrable functions on $\R_I^{n}$. Since ${{L}}^1 (\R_I^n) \subset {{L}}^1 (\R_I^{n+1})$, we define ${{L}}^1 [{{\R}'}_I^{\infty}]=\bigcup_{n=1}^{\iy}{{L}}^1(\R_I^n)$ and let ${{L}}^1 [\R_I^{\infty}]$ be the norm closure of ${{L}}^1 [{{\R}'}_I^{\infty}]$.  It follows that every function in ${{L}}^1 [\R_I^{\infty}]$ is the limit of a sequence of functions in ${{L}}^1 [{{\R}}_I^{n_k}]$, for some sequence $\{n_k\} \subset \N$.

Let $\C_c(\R_I^n)$ be the class of continuous functions on $\R_I^n$ which vanish outside compact sets.   We define $\C_c(\R_I^{\iy})$ to be the closure of $\bigcup_{n=1}^{\iy}\C_c(\R_I^n)=\C_c({{\R}'}_I^{\iy})$ in the $\sup$ norm. 
Thus, for any $f \in C_c(\R_I^{\iy})$, there always exists a sequence of functions $\{f_{n_k}\} \in \C_c({\R}_I^{n_k})$ such that $f_{n_k} \to f$, for some sequence $\{n_k\} \subset \N$.
We define $\C_0(\R_I^{\iy})$, the functions that vanish at $\iy$,  in the same manner.
\begin{lem} If $f \in C_c(\R_I^{\iy}) \; {\rm or} \;  C_0(\R_I^{\iy})$, then $f$ is continuous.
\end{lem}
\begin{proof}  Let $f({\bf x}) \in C_c(\R_I^{\iy})$ and let $\{{\bf x}_n\ | \ n=1,2,\dots \}$ be any sequence in $\R_I^n$ such that ${\bf x}_n \to {\bf x}$ as $ n\to \iy$.  If $\e>0$ is given, choose $K_1$ so that for $k \ge K_1$ and $f_k \in C_c(\R_I^{\iy}), \;  \lt|f_k({\bf x}_n)-f({\bf x}_n)\rt| <\tf{\e}{3}$.  Then choose $K_2$ so that for $k \ge K_2, \;  \lt|f_k({\bf x})-f({\bf x})\rt| <\tf{\e}{3}$.  Choose $N$ so that for $n \ge N, \;  \lt|f_k({\bf x}_n)-f_k({\bf x})\rt| <\tf{\e}{3}$.  If $n \ge N \; {\rm and} \; k \ge \max\{K_1,K_2\}$, we have: 
\[
\lt|f({\bf x}_n)-f({\bf x})\rt| \le \lt|f_k({\bf x}_n)-f({\bf x}_n)\rt| +\lt|f_k({\bf x})-f_k({\bf x}_n)\rt| +\lt|f_k({\bf x})-f({\bf x})\rt|<\e.
\]
The same proof applies to $ C_0(\R_I^{\iy})$
\end{proof}
\begin{thm} $\C_c(\R_I^{\iy})$ is dense in $L^1(\R_I^{\iy})$.
\end{thm}
\begin{proof}
We prove this result in the standard manner, by reducing the proof to positive simple functions and then to one characteristic function and finally using the approximation theorem to approximate a measurable set which contains a closed set and is contained in an open set.

First note that, since $\lim_{k \to \iy} \lt\|f \chi_{B_I(0,k)}-f\rt\|_1=0$ for all $f \in L^1$ (by the DCT),  we can prove the result for functions that vanish outside a compact set.  In this case, as $f=f_+ - f_-$, we need only consider positive $f$.  However, this function can be approximated by simple functions in $S_+$.  Since each simple function is a finite sum of characteristic functions (of bounded measurable sets) multiplied by finite constants, it follows that we need only show that we can approximate the characteristic function of a bounded measurable set by a continuous function which vanishes outside a compact set.  Let $\e>0$ be given and let $g=\chi_A$, where $A$ is any bounded measurable set.  By the regularity of $\la_{\iy}$, there exists an open set $O$ and a compact set $H$ with $H \subset A \subset O$ and $\la_{\iy}(O \setminus H) < \e$. 

Let $\{ V_n \}$ be the class of open intervals with rational end points. For each $n \in \N$, let $F_n \subset g^{-1}[V_n]$ and $G_n \subset (O \setminus g^{-1}[V_n])$ be compact sets, such that $\la_\iy[(O \setminus F_n \cup G_n)] < \tf{\e}{2^n}$.  If $H=\cap_{n=1}^\iy{[F_n \cup G_n]}$, then $\la_{\iy}(O \setminus H) < \e$.

If $x \in H$, there is an $n$ such that $f(x) \in V_n$ and $x \in G_n^c$, so that $g[G_n^c \cap H] \subset V_n$.  It follows that $g$ restricted to $H$ is continuous and $\la_{\iy}(A \setminus H) \le \la_{\iy}(O \setminus H) < \e$.
\end{proof}
In a similar fashion we can define the $L^p$ spaces, $1<p < \iy$.  We should note that, each space is defined relative to the family of indicator functions for $I$.  Thus, each space is the canonical one for that particular class of spaces.  
\section{Fourier Transform Theory}
In this section, we study the implications of Lebesgue measure on $\R^\iy$ for the Fourier transform and discuss two different extensions of the Pontrjagin Duality theory for Banach spaces. 
\subsection*{Background}
Let $G$ be a locally compact abelian (LCA) group  (c.f., $\R^n$).  The following is a restatement of Theorem 0.1 (see Rudin \cite{RU1}).
\begin{thm}If G is a LCA group and ${\mathfrak B}(G)$ is the Borel $\s$-algebra of subsets of $G$, then there is a non-negative regular translation invariant measure $\mu$ (i.e., 
 $\mu(g + A) =  \mu(A), \; A \in {\mathfrak B}(G$).  The (Haar) measure $\mu$  is unique up to multiplication by a constant.
\end{thm}
\begin{Def} A complex valued function $\al: G \to \C$ on a LCA group is called a character on $G$ provided that $\al$ is a homomorphism and $\lt| \al(g)\rt|=1$ for all $g \in G$.
\end{Def}  
The set of all continuous characters of $G$ defines a new group $\hat{G}$,
called the dual group of G and $(\al_1 + \al_2)(g)=\al_1(g) \cdot \al_2(g)$.  If we define a map $\g : G \to \hat{\hat{G}}$, by  $\g_g(\al)=\al(g)$, then the following theorem was proven by Pontryagin:
\begin{thm} {\rm (Pontryagin Duality Theorem)}
If $G$ is a LCA-group, then the mapping $\g : G \to \hat{\hat{G}}$ is an isomorphism of topological groups. 
\end{thm}
Thus, Pontrjagin Duality identifies those groups that are the character groups of their character groups.   If the group is not locally compact Theorem 5.1 does not hold (e.g., there is no Haar measure). However Kaplan \cite{KA1} has shown that the class of  topological abelian groups for which the Pontrjagin Duality holds is closed  under the operation of taking infinite products of groups. This result  immediately implies that  this class is larger than the class of locally compact abelian groups   because the infinite product of locally compact  groups (for example, $R^{\infty}$)  may be non-locally compact (see also \cite{KA2}).
\subsection{Pontryagin Duality Theory I}
In this section, we treat the Fourier transform as an operator.  As will be seen, this approach has the advantage of being constructive.  It also provides us with some insight into the problem that arises when we look at analysis on infinite dimensional spaces.

We define ${\mathfrak{F}}$ on $L^1[\R, \la]$ by
\[
{\hat{g}}(x)={\mathfrak{F}}(g)(x) =\int_{\mathbb{R} } {\exp \{- 2\pi ixy\} } g(y)dy.
\]
It is easy to check that  ${\mathfrak{F}}^{-1}$ is defined by
\[
g(y)={\mathfrak{F}}^{-1}({\hat{g}})(y) =\int_{\mathbb{R} } {\exp \{ 2\pi iyx\} } {\hat{g}}(x)dx.
\]
This representation is more convenient for the infinite-dimensional case, because we have no factors of $\sqrt{2\pi}$ to worry about.  

It is possible to  define ${\mathfrak{F}}$ as a mapping on $L^1[\R_I^n, \la]$ to $\C_0[\R_I^n, \la]$ for all $n$ as one fixed linear operator.  However, in the case of Hilbert spaces,Theorem 3.2(2) requires that we clearly specify our canonical domain and range space.  The same is also true for $L^1[\R_I^n, \la](h)$ and  $\C_0[\R_I^n, \la]({\hat{h}})$ (see \cite{GZ}). Since  $h=\otimes_{k=1}^\iy{\chi_I(x_k)}$, an easy calculation shows that ${\hat{h}}=\otimes_{k=1}^\iy{\tf{sin(\pi x_k)}{\pi x_k}}$.  Thus,  we can define  $\mathfrak{F}(f_n )({\bf x})$,  mapping  ${\mathbf{L}}^1[{\R}_I^{n}]({{h}})$ into ${\C}_0[{\R}_I^{n}]({\hat{h}})$ by
\beqn
{\mathfrak{F}}(f_n ) )({\bf x})  = { \otimes _{k = 1}^n \mathfrak{F}_k (f}_{(n)}) \otimes _{k = n + 1}^\infty  {\hat h}_k(x_k ).
\eeqn
\begin{thm}The operator $\mathfrak{F}$ extends to a bounded linear mapping of $L^1[{\R}_I^{\iy}]({{h}})$ into ${\C}_0[{\R}_I^{\iy}]({\hat{h}})$.
\end{thm}
\begin{proof}
Since    
\[
\mathop {\lim }\limits_{n \to \infty } L^1 [\mathbb{R}_I^n ]({{h}}) =  \bigcup _{n = 1}^\infty  L^1 [\mathbb{R}_I^n ]({{h}}) = L^1 [{\mathbb{R}'}_I^\infty ]({{h}})
\] 
and  $ L^1 [\mathbb{R}_I^\infty  ]({{h}})$ is the closure of  $L^1 [{\mathbb{R}'}_I^\infty  ]({{h}})$ in the  $L^1 {\text{ - norm}}$, it follows that $\mathfrak{F}$ is a bounded linear mapping of $L^1[{{\R}'}_I^{\iy}]({{h}})$ into ${\C}_0[{\R}_I^{\iy}]({\hat{h}})$.  

Supposed that $\{f_n \} \subset L^1 [{\mathbb{R}'}_I^\infty  ]({{h}})$, converges to $f \in L^1 [\mathbb{R}_I^\infty  ]({{h}})$.  Thus,  the sequence is Cauchy, so that $\left\| {f_n  - f_m } \right\|_1  \to 0$ as $m, \ n \to \infty$. It follows that
\[
\left| {\mathfrak{F}\left( {f_n ({\mathbf{x}}) - f_m ({\mathbf{x}})} \right)} \right| \leqslant \int_{\mathbb{R}_I^\infty  } {\left| {f_n ({\mathbf{y}}) - f_m ({\mathbf{y}})} \right|d\lambda _\infty  ({\mathbf{y}})}  = \left\| {f_n  - f_m } \right\|_1, 
\]
so that $\left| {\mathfrak{F}\left( {f_n ({\mathbf{x}}) - f_m ({\mathbf{x}})} \right)} \right|$ is a Cauchy sequence in ${\C}_0[{\R}_I^{\iy}]({\hat{h}})$. 
Since $L^1[{{\R}'}_I^{\iy}](h)$ is dense in $L^1[{\R}_I^{\iy}](h)$, it follows that $\mathfrak{F}$ has a bounded extension, mapping $L^1[{\R}_I^{\iy}](h)$ into ${\C}_0[{\R}_I^{\iy}]({\hat{h}})$.
\end{proof}
\subsection{$L^2$-Theory}
In the case of $L^2$, the Fourier transform is an isometric isomorphism from $L^2[{\R}^{n}]$ onto $L^2[{\R}^{n}]$.     
\begin{thm} The operator $\mathfrak{F}$ is an isometric isomorphism of ${{L}}^2[{\R}_I^{\infty}]({h})$ onto  ${{L}}^2[{\R}_I^{\infty}](\hat{h})$.
\end{thm}
\begin{proof}  Let $f \in {{L}}^2[{\R}_I^{\infty}]({h})$.  By construction, there exists a sequence of functions $\{f_k \in {{L}}^2[{\R}_I^{n_k}], \, n_k \in \N\}$ such that $
\mathop {\lim }\limits_{k \to \infty } \left\| {f - f_k } \right\|_2({h})  = 0$.  Furthermore, since the sequence converges, it is a Cauchy sequence.  Hence, given $\e >0$, there exists a $N(\e)$ such that $m,\,n \ge N(\e)$ implies that $\left\| {f_m - f_n } \right\|_2({h}) < \e$.  Since $\mathfrak{F}$ is an isometry, $\left\| {\mathfrak{F}(f_m) - \mathfrak{F}(f_n) } \right\|_2(\hat{h}) < \e$, so that the sequence $\mathfrak{F}(f_k)$ is also a Cauchy sequence in ${{L}}^2[{\R}_I^{\infty}](\hat{h})$.  Thus, there is a $\hat{f} \in {{L}}^2[{\R}_I^{\infty}](\hat{h})$ with $ \mathop {\lim }\limits_{k \to \infty }\left\| {\hat{f}  - \mathfrak{F}(f_k) } \right\|_2(\hat{h})=0$, and we can define $\mathfrak{F}(f)=\hat{f}$.  It is easy to see that $\hat{f}$ is unique.
\end{proof}

We can also prove a version of Theorems 5.4 and 5.5 for every separable Banach space (with a basis).  Fix  $\mcB$ and for each $n$, let ${\mcB}_I^n ={\mcB}\cap \R_I^n$.   It is  clear that ${\mcB}_I^n \subset {\mcB}_I^{n+1}$, so that $\mcB$ is the norm closure of $\mathop {\lim }\limits_{n \to \infty } \mcB_I^n$. The following have the same proofs as  Theorems 5.4 and 5.5.   
\begin{thm}The operator $\mathfrak{F}$ extends to a bounded linear mapping of $L^1[{\mcB}]({h})$ into ${\C}_0[{\mcB}](\hat{h})$.
\end{thm}
\begin{thm} The operator $\mathfrak{F}$ is an isometric isomorphism from ${{L}}^2[{\mcB}]({h})$ onto ${{L}}^2[{\mcB}](\hat{h})$.
\end{thm}

Theorems 5.4 - 5.7 show that $\otimes_{i=1}^\iy {\hat {h}}_i$ is a strong (product) convergence vector for the Fourier transform operator ${\mathfrak{F}}$.  In the $L^2$-theory, we know that ${{L}}^2[{\R_I^\iy}]({h})$ and ${{L}}^2[{\R_I^\iy}](\hat{h})$ are orthogonal subspaces of $\mcH_{\otimes}^2$.  Thus, in this approach, the natural interpretation is that the Fourier transform induces a Pontryagin duality like theory that does not depend on the group structure of $\R_I^\iy$, but depends on the pairing of different function spaces.  This approach is direct, constructive and applies to all separable Banach spaces (with a basis).   Thus, the group structure of the underlying measure space plays no role. 
\subsection{Pontryagin Duality Theory II}
In this section, we show that the standard form of Pontryagin duality  theory is also possible, using the underlying measure space group structure. It is constructive but restrictive, in that, it does not apply to every separable Banach space with a basis. 

Let $\mcB$ be any uniformly convex separable Banach space UCB over the reals, so that $\mcB = {\mcB}^{**}$ (second dual). The next  theorem follows from our theory of Lebesgue measure on Banach spaces.
\begin{thm}If $\la_{\mcB}$ is our version of Lebesgue measure on ${\mcB}$, then $\mcB$ and ${\mcB}^{*}$ are also duals as character groups (i.e., ${\mcB}^{*}=\hat{\mcB}$).
\end{thm} 
\begin{proof} If we consider the restriction to ${{L}}^2[\mcB, \la_{\mcB}]$,  we can define ${\mathfrak{F}}$ directly by: 
\beqn
[{\mathfrak{F}}(f)]({\mathbf{x}}^*)=\hat f({\mathbf{x}}^*) = \int_\mcB {\exp \{  - 2\pi i\left\langle {{\mathbf{y}},{\mathbf{x}}^*} \right\rangle \} f({\mathbf{y}})d\lambda_{\mcB} ({\mathbf{y}})} ,
\eeqn
where ${\left\langle {{\mathbf{y}},{\mathbf{x}}^*} \right\rangle }$ is the pairing between $\mcB$ and ${\mcB}^{*}$.  From Plancherel's Theorem, we have:
\[
\left\| {\hat f} \right\|_2^2  = \left( {\hat f,\hat f} \right)_2  = \left( {f,f} \right)_2  = \left\| f \right\|_2^2 .
\]
It follows that  $\mcB$ and ${\mcB}^{*}$ are duals as character groups and 
\[
f({\mathbf{y}}) = \int_{{\mcB}^{*}} {\exp \{ 2\pi i\left\langle {y,{\mathbf{x}}^* } \right\rangle \} \hat f({\mathbf{x}}^* )d\lambda_{{\mcB}^*} ({\mathbf{x}}^* )} .
\]
\end{proof}
If ${\mcB}_I^n ={\mcB}_I\cap \R_I^n$, we can represent $\hat{f}_n$ directly as a mapping from ${{L}}^2[\mcB_I^n, \la_{\mcB}] \to {{L}}^2[{\mcB}_I^{*,  n}, \lambda_{{\mcB}^*}]$, by
\beqa
[{\mathfrak{F}}(f_n)]({\mathbf{x}}^*)=\hat f_n({\mathbf{x}}^*) = \int_\mcB {\exp \{  - 2\pi i\left\langle {{\mathbf{y}},{\mathbf{x}}^*} \right\rangle_n \} f_n({\mathbf{y}})d\lambda_{\mcB} ({\mathbf{y}})} ,
\eeqa
where $\left\langle {{\mathbf{y}},{\mathbf{x}}^*} \right\rangle_n$ is the restricted pairing of ${\mathbf{y}}$ and ${\mathbf{x}}^*$ to $\mcB_I^n$ and ${\mcB}_I^{*,  n}$ respectively.  It follows that equations 5.1 and 5.2 provide two distinct definitions of the Fourier transform for ${\mcB}$. Thus, in this approach the  group structure of the underlying measure space changes.

It is clear that representation for $\hat f({\mathbf{x}}^*)$ also applies if we use ${{L}}^1[\mcB, \la_{\mcB}]$, but in this case $\hat f({\mathbf{x}}^*) \in \C_0[{\mcB}^*]$.

If we define ${\mathbf{y}}(\cdot)$ mapping ${\mcB} \to \C$, by ${\mathbf{y}}({\mathbf{x}}) =exp\{ -2\pi i \left\langle {{\mathbf{y}},{\mathbf{x}}^*} \right\rangle \}$, then ${\mathbf{y}}({\mathbf{x}})$ is a character of ${\mcB}$.  Furthermore, it is easy to see that $({\mathbf{y}}_1 + {\mathbf{y}}_2)({\mathbf{x}})={\mathbf{y}}_1({\mathbf{x}}) \cdot {\mathbf{y}}_2({\mathbf{x}})$. We now have the extension of the Pontryagin Duality Theorem to all UCB (with a basis).
\begin{thm}If ${\mcB}$ is a UCB,  then the mapping $\g_{\mathbf{x}} : {\mcB} \to \hat{\hat{{\mcB}}}$, defined by $\g_{\mathbf{x}}({\mathbf{y}})= {\mathbf{y}}({\mathbf{x}})$, is an isomorphism of topological groups. 
\end{thm}
In case $\mcB= {\mcH}$, is a Hilbert space, we  can replace equation (5.2) by  
\beqn
\hat f({\mathbf{x}}) = \mathfrak{F}[f]({\mathbf{x}}) = \int_{\mcH } {\exp \{  - 2\pi i \left\langle {{\mathbf{x}},{\mathbf{y}}} \right\rangle_{\mcH} \} f({\mathbf{y}})d\lambda _{\mcH} ({\mathbf{y}})},
\eeqn
so that $\mcH $ is self-dual (as expected).  From equation (5.3), we also get the expected result that: 
\[
\mathfrak{F}\left[ {\exp \{  - \pi \left| {\bf x} \right|_{\mcH}^2 \} } \right] = \exp \{  - \pi \left| {\mathbf{x}} \right|_{\mcH}^2 \} .
\]

In closing, we observe that by Theorem 2.31 (see Example  2.32), if we use Gaussian measure on  ${\R}^{\infty}$,  the dual character groups are ${\R}^{\infty}$ and ${\hat{\R}}^{\infty} ={\R}_0^{\infty}$.  From this we see that probability  measures on $({\R}^{\infty}, \; \mathfrak{B}[{\R}^{\infty}])$ induce a different character theory compared to that induced by $\la_\iy$  on $({\R}_I^{\infty}, \; \mathfrak{B}[{\R}_I^{\infty}])$.
\subsection{$L^p$-Theory}
We can obtain $L^p[\R_I^\iy]$ as in the construction of $L^1[\R_I^\iy]$.  In this section we want to show the power of our approach to measure theory by establishing a version of Young's Theorem for every separable Banach space with a Schauder basis:
\begin{thm}  {\rm (Young)} Let $p,q,r \in [1, \iy]$ with
\[ 
\f{1}{r} =\f{1}{p} +\f{1}{q} -1.
\]
If $f \in L^p[\R_I^\iy]$ and $g \in L^q[\R_I^\iy]$, then the convolution of $f$ and $g, \; f * g$, exists  (a.s.), belongs to $L^r[\R_I^\iy]$ and
\[
\left\| {f * g} \right\|_r  \leqslant \left\| f \right\|_p \left\| g \right\|_q .
\]
\end{thm}
\begin{cor}Let $\mcB$ be a separable Banach space with a Schauder basis and let $p,q,r \in [1, \iy]$ with
\[ 
\f{1}{r} =\f{1}{p} +\f{1}{q} -1.
\]
If $f \in L^p[\mcB]$ and $g \in L^q[\mcB]$, then the convolution of $f$ and $g, \; f * g$, exists  (a.s.), belongs to $L^r[\mcB]$ and
\[
\left\| {f * g} \right\|_r  \leqslant \left\| f \right\|_p \left\| g \right\|_q .
\]
\end{cor}
In order to prove Theorem 5.10, we first need the appropriate version of Fubini's Theorem.  Since  $\left( {\mathbb{R}_I^\infty  ,\mathfrak{L}_I^\infty  ,\lambda _\infty  } \right)$ is a complete $\s$-finite measure space, a proof of the following may be found in Royden \cite{RO} (see Theorems 19 and 20, pgs. 269-270):
\begin{thm} {\rm (Fubini)} If $f \in L^1[\R_I^\iy \times \R_I^\iy]$, then
\begin{enumerate}
\item for almost all $x \in \R_I^\iy$ the function $f_x$ defined by $f_x(y)=f(x,y) \in  L^1[\R_I^\iy](y)$:
\item for almost all $y \in \R_I^\iy$ the function $f_y$ defined by $f_y(x)=f(x,y) \in  L^1[\R_I^\iy](x)$:
\item $\int_{\mathbb{R}_I^\infty  } {f(x,y)d\lambda _\infty  (y)}  \in L[\mathbb{R}_I^\infty  ](x)$;
\item $\int_{\mathbb{R}_I^\infty  } {f(x,y)d\lambda _\infty  (x)}  \in L[\mathbb{R}_I^\infty  ](y)$;
\item
\[
\begin{gathered}
  \int_{\mathbb{R}_I^\infty   \times \mathbb{R}_I^\infty  } {f(x,y)d\left( {\lambda _\infty   \otimes \lambda _\infty  } \right)(x,y)}  \hfill \\
   = \int_{\mathbb{R}_I^\infty  } {\left[ {\int_{\mathbb{R}_I^\infty  } {f(x,y)d\lambda _\infty  (y)} } \right]d\lambda _\infty  (x)}  = \int_{\mathbb{R}_I^\infty  } {\left[ {\int_{\mathbb{R}_I^\infty  } {f(x,y)d\lambda _\infty  (x)} } \right]d\lambda _\infty  (y)} . \hfill \\ 
\end{gathered}
\] 
\end{enumerate}
\end{thm}
\begin{thm}Let   $f,g \in L^1[\R_I^\iy]$, then $(f*g)(x)$ exists (a.s.); that is $f(y)g(x-y) \in L^1[\R_I^\iy]$.  In addition,  $f*g \in L^1[\R_I^\iy]$ and 
\[
\left\| {f * g} \right\|_1  \leqslant \left\| f \right\|_1 \left\| g \right\|_1 .
\]
\end{thm}
\begin{proof}  First, it is easy to see that $f(y)g(x-y)$ is a measurable function on $\R_I^\iy$. (There is no change from the case of $\R^n$.)  We can apply Fubini's theorem to get that:
\[
\begin{gathered}
  \int_{\mathbb{R}_I^\infty  } {\left( {f * g} \right)(x)d\lambda _\infty  (x)}  \hfill \\
   = \int_{\mathbb{R}_I^\infty  } {d\lambda _\infty  (x)\left[ {\int_{\mathbb{R}_I^\infty  } {f(y)g(x - y)d\lambda _\infty  (y)} } \right]}  = \int_{\mathbb{R}_I^\infty  } {d\lambda _\infty  (y)\left[ {\int_{\mathbb{R}_I^\infty  } {f(y)g(x - y)d\lambda _\infty  (x)} } \right]}  \hfill \\
  \quad \quad \quad \quad \quad \quad  = \int_{\mathbb{R}_I^\infty  } {f(y)d\lambda _\infty  (y)}  \cdot \int_{\mathbb{R}_I^\infty  } {g(x)d\lambda _\infty  (x)} . \hfill \\ 
\end{gathered} 
\]
It follows from the last equality that $\left\| {f * g} \right\|_1  \leqslant \left\| f \right\|_1 \left\| g \right\|_1$.
\end{proof}
\subsubsection{\rm Proof of Young's Theorem}
\begin{proof}  First, assume that $f$ and $g$  are nonnegative and $\left\| f \right\|_p  = \left\| g \right\|_q  = 1$. Let $\tf{1}{q'} =1-\tf{1}{q}$ and $\tf{1}{p'} =1-\tf{1}{p}$.  Now note that
\[ 
\begin{gathered}
  \frac{1}
{r} + \frac{1}
{{q'}} + \frac{1}
{{p'}} = \frac{1}
{r} + \left( {1 - \frac{1}
{q}} \right) + \left( {1 - \frac{1}
{p}} \right) = 1; \hfill \\
  \left( {1 - \frac{p}
{r}} \right)q' = p\left( {\frac{1}
{p} - \frac{1}
{r}} \right)q' = p\left( {1 - \frac{1}
{q}} \right)q' = p; \hfill \\
  \left( {1 - \frac{q}
{r}} \right)p' = q\left( {\frac{1}
{q} - \frac{1}
{r}} \right)p' = q\left( {1 - \frac{1}
{p}} \right)p' = q. \hfill \\ 
\end{gathered} 
\]
If we use Holder's inequality (for three functions), we can write $(f*g)(x)$ as:
\[
\begin{gathered}
  \left( {f * g} \right)(x) = \int_{\mathbb{R}_I^\infty  } {\left[ {f(y)^{p/r} g(x - y)^{q/r} } \right]\left[ {f(y)^{1 - p/r} g(x - y)^{1 - q/r} } \right]d\lambda _\infty  (y)}  \hfill \\
   \leqslant \left[ {\int_{\mathbb{R}_I^\infty  } {f(y)^p g(x - y)^q d\lambda _\infty  (y)} } \right]^{1/r} \left[ {\int_{\mathbb{R}_I^\infty  } {f(y)^{\left( {1 - p/r} \right)q'} d\lambda _\infty  (y)} } \right]^{1/q'} \left[ {\int_{\mathbb{R}_I^\infty  } {g(x - y)^{\left( {1 - q/r} \right)p'} d\lambda _\infty  (y)} } \right]^{1/p'} . \hfill \\ 
\end{gathered} 
\]
This last inequality shows that
\[
\begin{gathered}
  \left( {f * g} \right)(x) \leqslant \left[ {\int_{\mathbb{R}_I^\infty  } {f(y)^p g(x - y)^q d\lambda _\infty  (y)} } \right]^{1/r}   \Rightarrow  \hfill \\
  \left( {f * g} \right)^r (x) \leqslant \left[ {\int_{\mathbb{R}_I^\infty  } {f(y)^p g(x - y)^q d\lambda _\infty  (y)} } \right]\quad  \Rightarrow \left( {f * g} \right)^r (x) \leqslant \left( {f^p  * g^q } \right)(x). \hfill \\ 
\end{gathered} 
\]
From Theorem 5.13, we have $
\left\| {\left( {f * g} \right)^r } \right\|_1  \leqslant \left\| {f^p } \right\|_1 \left\| {g^q } \right\|_1  = 1$.
In the general case, we know that $\left| f \right| * \left| g \right|$ exists (a.e.), so that $\left| {f(y)g(x - y)} \right| \in L^1 [\mathbb{R}_I^\infty  ]$.  But then, $f(y)g(x - y) \in L^1 [\mathbb{R}_I^\infty  ]$.
\end{proof}
In closing we note that, Beckner \cite{BE} and Brascamp-Lieb \cite{BL} have shown that on $\R^n$ we can write Young's inequality as  $\left\| {f * g} \right\|_r  \leqslant (C_{p,q,r;n})^n \left\| f \right\|_p \left\| g \right\|_q$, where $C_{p,q,r;n} \le1$ is sharp.  We conjecture that $1$ is the sharp constant for $\R_I^\iy$.
\section{Partial Differential Operators (Examples)}  
In this section, we give examples of strong product and sum vectors for differential operators that have found interest in infinite dimensional analysis.  

\begin{Def} For $x \in \R, \ 0 \le y < \iy$ and $1<a< \iy$ define $\bar{g}(x, y), \ \bar{h}(x)$ by:
\[
\begin{gathered}
  \bar{g}(x,y) = \exp \left\{ { - y^a e^{iax} } \right\}, \hfill \\
  \bar{h}(x) = \left\{ {\begin{array}{*{20}c}
   \begin{gathered}
  \int_0^\infty  {\bar{g}(x,y)dy} ,\;x \in [ - \tfrac{\pi }
{{2a}},\tfrac{\pi }
{{2a}}], \hfill \\
  0 \quad \quad {\text{            {\rm otherwise} }}. \hfill \\ 
\end{gathered}   \\
   {}  \\

 \end{array} } \right. \hfill \\ 
\end{gathered} 
\]
\end{Def}
The following properties of $\bar{g}$ are easy to check:
\begin{enumerate}
\item 
\[
\frac{{\partial \bar{g}(x,y)}}
{{\partial x}} =  - iay^a e^{iax} \bar{g}(x,y),
\]
\item
\[
\frac{{\partial \bar{g}(x,y)}}
{{\partial y}} =  - ay^{a - 1} e^{iax} \bar{g}(x,y),
\]
so that
\item
\[
iy\frac{{\partial \bar{g}(x,y)}}
{{\partial y}} = \frac{{\partial \bar{g}(x,y)}}
{{\partial x}}.
\]
\end{enumerate}
It is also easy to see that $\bar{h}(x)$ is in ${\bf{L}}^1[\R]$ for $x \in [- \tf{\pi}{2a}, \tf{\pi}{2a}]$ and,
\beqn
\frac{{d\bar{h}(x)}}
{{dx}} = \int_0^\infty  {\frac{{\partial \bar{g}(x,y)}}
{{\partial x}}dy}  = \int_0^\infty  {iy\frac{{\partial \bar{g}(x,y)}}
{{\partial y}}dy}.
\eeqn
Integration by parts in the last expression of equation (6.1) shows that $\bar{h}'(x) =  - i\bar{h}(x)$, so that $
\bar{h}(x) = \bar{h}(0)e^{ - ix}$ for $x \in [- \tf{\pi}{2a}, \tf{\pi}{2a}]$.  Since $
\bar{h}(0) = \int_0^\infty  {\exp \{  - y^a \} dy}$, an additional integration by parts shows that $\bar{h}(0)= \G(\tf{1}{a} +1)$. 

Let $a=\tf{\pi}{1-\e},\; \bar{h}(x)=\bar{h}_{\e}(x), \ x \in [- \tf{\pi}{2a}, \tf{\pi}{2a}]$, where $
0 < \varepsilon  \ll 1$, and define
\beqn
f_{\e} (x) = \left\{ {\begin{array}{*{20}c}
   {c \exp \left\{ {\frac{{\varepsilon^2 }}
{{\left| {2x }\right|^2- \varepsilon^2 }} }\right\},\quad \left| {x } \right| < \varepsilon/2 ,}  \\
   {0,\quad \quad \quad \quad \quad \quad \quad \quad \quad \quad \quad \left| {x} \right| \geqslant \varepsilon/2 ,}  \\

 \end{array} } \right.
\eeqn
where $c$ is the standard normalizing constant.  We now define $\xi{(x)}= (\bar{h}*f_{\e})(x)$, so that ${\rm{spt}}(\xi)=[-\tf{1}{2}, \tf{1}{2}]=I_\e$.  Thus, $\xi{(x)}=0,\;x \notin I_\e$ and otherwise,  
\[
\xi{(x)}  = \int_{ - \infty }^\infty  {\bar{h} [x - z]} f_{\e} (z)dz = e^{ - ix} \int_{ - \infty }^\infty  {e^{iz} f_{\e} (z)dz}= \al_\e e^{-ix}. 
\]
It follows from this that: 
\[
\al_\e^{-1}\xi (ix) = \left\{ {\begin{array}{*{20}c}
   {e^x ,\;x \in I_\e  }  \\
   {0,\;\,x \notin I_\e  }  \\

 \end{array} } \right.
\]
Define $\la_\e=\la$ and,
\[
I^\varepsilon   =  \times _{k = 1}^\infty  I_\varepsilon
,\;I_n^\varepsilon   =  \times _{k = n + 1}^\infty  \; {\rm and}, \;\lambda _\infty ^\varepsilon   =  \otimes _{k = 1}^\infty  \lambda _\varepsilon.  
\] 
\begin{ex}
In this example, we let  $h_k(x_k)=\al_\e^{-1}\xi (ix_k)$, for each $k \in \N$ so that $D_k h_k=h_k \; $, for $x_k \in I$. Let $L_\e^2[\mathbb{R}_I^n]= L^2[\mathbb{R}_{I}^n, \la_\iy^\e]$.  If  ${\bf D}^{\iy} = \prod _{k =1}^{\iy}{D_k}$ and  $f_n \in L_\e^2[\mathbb{R}_I^n] \cap D({\bf D}^{\iy})$, we can define ${\bf D}^{\iy}$ on $\mathbb{R}_{I}^n$ by  ${\bf D}^{\iy} f_n(x)={\bf D}^{n} f_n(x)=\prod _{l =1}^{n} D_lf_{(n)}(x)\otimes\lt(\otimes_{l=n+1}^{\iy}h_l\rt), \; (a.s)$. This operator  is well-defined and has a closed densely defined extension to $L_\e^2[\mathbb{R}_I^\infty](h)$, where $h=\otimes_{k=1}^{\iy}{h_k}$.  Thus, $h$ is a strong product vector for  ${\bf D}^{\iy}$.

The operator ${\bf D}^{\iy}$ is required if we want to obtain the probability density for a  distribution function.  (Note, by construction the density can be approximated from below by densities in a finite number of variables.)
\end{ex} 
In the following example, we construct a general elliptic operator on $L_\e^2[\mathbb{R}_I^\infty]$.
\begin{ex}
If \, $\nabla = \lt( D_1,D_2, \cdots  \rt)$ and  $\s_k: \R_I^{\iy} \to \R$ is a bounded analytic function for each $k \in \N$, then let ${\boldsymbol \s}(x)=\lt( \s_1(x),\s_2(x), \cdots \rt)$.  
We assume that
\[
\left\| {\sum\limits_{j,k = 1}^\infty  {\s_j(x) \s_k(x) }  + \sum\limits_{k = 1}^\infty  {b_k(x) } } \right\|_2  < \infty ,\;{\text{where}}\;b_k(x)  = \sum\nolimits_{j = 1}^\infty  {\s_j(x) D_j \s_k(x) } .
\]
We can now define $\Delta _\infty$ by: 
\[
\Delta _\infty   = \left( {{\boldsymbol \s}(x) \cdot \nabla } \right)^2  = \sum\limits_{j,k = 1}^\infty  {\s_j(x) \s_k(x) D_j D_k }  + \sum\limits_{k = 1}^\infty  {b_k(x) } D_k. 
\]
For the same version of $L_\e^2[\mathbb{R}_I^{\iy}]$ as in the last example, if $g_n \in L^2[\mathbb{R}_I^n] \cap D(\De_{\iy})$ and $c_n(x)  = \sum\nolimits_{j = n+1}^\infty  {\s_j(x) \left[ {D_j \s_k(x) } \right]}$,  then $\Delta _{\iy}$ is defined on $\mathbb{R}_I^n$  and  
\[
\De_{\iy} g_n(x) = \sum\limits_{j,k = 1}^n {\s_j(x)  \s_k(x)  D_j  D_k  g_n(x)}  + \sum\limits_{k = 1}^n {b_k } D_k \psi(x) +c_n(x) g_n(x).
\]
In order to obtain the same equation with $c_n(x)=0$,  we use the version of $L_\e^2[\mathbb{R}_I^\infty]$ defined with $h_k= \xi_I(0), \ k \ge n+1$. In this case, for $g_n(x) \in L_\e^2[\mathbb{R}_I^\infty] \cap D(\De_\infty)$, we see  that
\[
\De_{\iy} g_n(x) = \sum\limits_{j,k = 1}^n {\s_j(x)  \s_k(x)  D_j  D_k  g_n(x)}  + \sum\limits_{k = 1}^n {b_k(x) } D_k g_n(x).  
\]
In either case,  $\De_{\iy}$ is well-defined for each $n$ and has a closed densely defined extension to $L_\e^2[\mathbb{R}_I^\infty]$ and $g_n(x) \to g(x)$ implies that $lim_{n \to \iy}\De_{\iy} g_n(x) =\De_{\iy} g(x)$.  
 
It follows that different versions of $L^2[\mathbb{R}_I^\infty]$ offer advantages for particular types of differential operators.  (For other approaches, see \cite{BK},  \cite{GZ} and \cite{UM}.)
\end{ex}
The following special cases have appeared in the literature (all can be obtained from the last example):
\begin{enumerate}
\item
The natural infinite dimensional Laplacian:
\[
{\mathbf{A}} = \Delta _\infty   = \sum\nolimits_{i = 1}^\infty  {{{\partial ^2 } \mathord{\left/
 {\vphantom {{\partial ^2 } {\partial x_i^2 }}} \right.
 \kern-\nulldelimiterspace} {\partial x_i^2 }}} .
\]
\item
The nonterminating diffusion generator in infinitely many variables (also known as the Ornstein-Uhlenbeck operator):
\[
{\mathbf{A}} = \tfrac{1}
{2}\Delta _\infty   - B{\mathbf{x}} \cdot \nabla _\infty   = \tfrac{1}
{2}\sum\nolimits_{i = 1}^\infty  {{{\partial ^2 } \mathord{\left/
 {\vphantom {{\partial ^2 } {\partial x_i^2 }}} \right.
 \kern-\nulldelimiterspace} {\partial x_i^2 }}}  - \sum\nolimits_{i = 1}^\infty  {b_i x_i {\partial  \mathord{\left/
 {\vphantom {\partial  {\partial x_i }}} \right.
 \kern-\nulldelimiterspace} {\partial x_i }}}.
\]
The infinite dimensional Laplacian of Umemura \cite{UM}:
\[
{\mathbf{A}} = \sum\limits_{i = 1}^\infty  {\left( {\frac{{\partial ^2 }}
{{\partial x_i^2 }} - \frac{{x_i }}
{{c^2 }}\frac{\partial }
{{\partial x_i }}} \right)}.
\]
\end{enumerate}
Berezanskii and Kondratyev (\cite{BK}, pages 520-521) have also discussed operators analogous to (2) and (3). 
\subsection{Discussion}
In a very interesting paper, Phillip Duncan Thompson \cite{PDT}  used the amplitudes of a set of orthogonal modes as the co-ordinates in an infinite-dimensional phase space. This allowed him to derive the probability distribution for an ensemble of randomly forced two-dimensional viscous flows as the solution of the continuity equation for the phase flow.  He obtained the following equation for the probability density:
\beqn
 \frac{{\partial \rho{} }}
{{\partial t}} + \sum\limits_{k = 1}^\infty  {M_k ({\mathbf{x}})\frac{{\partial \rho{} }}
{{\partial x_k }}}  - \nu \sum\limits_{k = 1}^\infty  {\frac{\partial }
{{\partial x_k }}\left[ {\rho{} \alpha _k^2 x_k } \right]}  - \sum\limits_{k = 1}^\infty  {\frac{{\partial ^2 \rho{} }}
{{\partial x_k^2 }}}  = 0,
\eeqn
where 
\[
M_k ({\mathbf{x}}) = \sum\limits_{i = 1}^\infty  {\sum\limits_{j = 1}^\infty  {\frac{{\alpha _j^2 \beta _{ijk} }}
{{\alpha _i \alpha _j \alpha _k }}\frac{{\left( {\mu _i \mu _j \mu _k } \right)^{\tfrac{1}
{2}} }}
{{\mu _k }}x_i x_j } }. 
\]
The coefficients $\beta_{ijk}$ vanishes if any two indices are equal, is invariant under cyclic permutation of indices and reverses sign under non-cyclic permutation of indices, while the coefficients $\al_i$ and $\mu_i$ are positive constants, determined by the problem.  Thompson imposed the natural condition
\beqn
\int\limits_{ - \infty }^\infty  { \cdots \int\limits_{ - \infty }^\infty  {\rho ({\mathbf{x}},t)\prod\limits_{k = 1}^\infty  {dx_k } } }  = 1.
\eeqn
At that time, he ran into the obvious mathematical criticism because equation (6.4) was meaningless at the time.  He also derived the equilibrium density
\beqn
\rho _0 ({\mathbf{x}}) = C\exp \left\{ { - \tfrac{1}
{2}\nu \sum\limits_{k = 1}^\infty  {\alpha _k^2 x_k^2 } } \right\}.
\eeqn
The results in  section 2.5, see also equation (2.5), along with those in section 4.4, show that Thomson's paper was prescient.
\section{Conclusion}
In this paper we provided a reasonable version of Lebesgue measure on $R^{\infty}$, which together with the standard Gaussian measure on $R^{\infty}$, have allowed us to construct natural analogues of Lebesgue and Gaussian measure for every separable Banach space with a Schauder basis.  We have extended the Fourier transform to  $L^1[\R^{\iy}, \la_{\iy}], \; L^2[\R^{\iy}, \la_{\iy}]$, defined sums and products of unbounded operators, and presented a few constructive examples of partial differential operators in infinitely many variables.
\subsection*{Acknowledgments}
This work could not have been written without the generous help and critical remarks of Professor Frank Jones.  We thank Professor Anatoly Vershik for appraising us of recent work by the Russian school and correcting some of our historical remarks.  We would also like to thank an anonymous referee for corrections that have improve our presentation and for suggesting that we reconsider our approach to the Fourier transform and discuss its relationship to the Pontryagin duality theory, which led to a complete revision and extension of Section 5.

\end{document}